\documentclass[11pt,a4paper]{amsart}
\usepackage{amsmath,amssymb,amsthm,mathtools}
\usepackage{hyperref}
\usepackage[margin=1in]{geometry}
\usepackage[T1]{fontenc}
\usepackage[style=alphabetic]{biblatex}
\usepackage{etoolbox}

\newtheorem{theorem}{Theorem}
\newtheorem{lemma}{Lemma}[theorem]
\newtheorem{proposition}{Proposition}
\newtheorem{claim}{Claim}

\allowdisplaybreaks 

\nocite{*}

\title{The Existence and Uniqueness of a Nash Equilibrium in Mean Field Game Theory}
\author{Daniel Block, Moises Reyes Rivas}
\date{July 28, 2022}

\begin{document}
\maketitle

\tableofcontents
\section*{Acknowledgments}
This work arises from the David Harold Blackwell Summer Research Institute conducted at the University of California, Los Angeles during the summer of 2022. Special appreciation to Dr. Wilfrid Gangbo of the UCLA Department of Mathematics for the generous investment of their time and patience.
\section{Introduction}

In recent and past works, convexity is usually assumed on each individual part of the action functional in order to demonstrate the existence and uniqueness of a Nash equilibrium on some interval $[0,T]$ (this meant that each hessian was assumed to be nonnegative). Particularly, a certain assumption was imposed in order to quantify the smallness of $T$. The contribution of this project is to expand on this with the key insight being that one does not need the convexity of each part of the action, but rather just an appropriate combination of them, which will essentially ``compensate'' for the other two terms to yield convexity in the action.

\smallskip

This is meaningful in both the pure and applied settings as it generalizes the existence and uniqueness of a Nash equilibrium slightly more, but maybe more importantly matches real-world application slightly closer, as in reality there are many settings in which not each part of the action have convexity. Thus, it is more accurate for modern application of Mean Field Game Theory.
\section{Initial Assumptions}

We have Lagrangian $L$ and Hamiltonian $H$ where $L,H \in C^{3}(\mathbb{R}^{2d})$ which represent one individual player's cost at a specific time using either their velocity (control), or momentum, respectively. $L$ and $H$ are such that 
\begin{equation*}
    L^{*}(q, \cdot) = H(q, \cdot) \quad \text{and} \quad H^{*}(q, \cdot) = L(q, \cdot)
\end{equation*}

where $f^{*}$ denotes the Legendre transform of $f$. We assume $L \geq 0$ and that there exist $c_0,c_1 > 0$ such that
\begin{equation*}
    c_0|v|^2 \leq L(q,v) \quad \text{and} \quad |\nabla_{q} L(q,v)| \leq c_1.
\end{equation*}

We also assume that for some interpolation (convex combination of two variables) of $L$, that the second derivative is bounded: that is, for some $0 \leq \varepsilon \leq 1$, we have some $\lambda_0 > 0$ and $q_0,q_1,v_0,v_1 \in \mathbb{R}^d$,
\begin{equation*}
    \lambda_0|v_1 - v_0|^2 \leq \frac{d^2}{d\varepsilon^2}L\left( (1 - \varepsilon)q_0 + \varepsilon q_1, (1 - \varepsilon)v_0 + \varepsilon v_1 \right).
\end{equation*}

Additionally, we have \emph{interaction} functions $\Phi, \Psi \in C^{3}(\mathbb{R}^d)$ that model the interaction between players labelled $\omega$ for $ \omega \in \Omega = (0,1)$ either at some random time or at initial time, respectively. We assume that both $\Phi$ and $\Psi$ are even, as well as being bounded from below with bounded first, second and third derivatives: let $c_{\Phi}$ and $c_{\Psi}$ be such that
\begin{equation*}
    c_{\Phi}I_d \leq \nabla^{2}\Phi \quad \text{and} \quad c_{\Psi}I_d \leq \nabla^{2}\Psi.
\end{equation*}

Denote $c^{+}$ and $c^{-}$ as the positive and negative parts of a real number $c$, i.e.,
\begin{align*}
    c^{+} &\coloneqq \max\{c,0\} \\
    c^{-} &\coloneqq \max\{-c,0\} \\
    c &= c^{+} - c^{-}.
\end{align*}

We suppose that for terminal time $T > 0$, 
\begin{equation*}
    T^2 \left(c^{-}_{\Phi} + 2\|\nabla^2 \Phi\|_{L^\infty}\right) + T \left(c^{-}_{\Psi} + 2\|\nabla^2 \Psi\|_{L^\infty}\right) < \frac{\lambda_0}{2}
\end{equation*}

Now set $\mathbb{H} = L^2\left((0,1), \mathcal{L}^1_{(0,1)}; \mathbb{R}^d\right)$ (the space of square-integrable, Lebesgue measurable functions mapping from $(0,1) \to \mathbb{R}^d$) and set $\Omega = (0,1)$. We can now define initial cost functional $\mathcal{U}_{*} : \mathbb{H} \to \mathbb{R}$ and running cost functionals, $\hat{L}, \hat{H} : \mathbb{H}^2 \to \mathbb{R}$ which take in all player trajectories $X$ and velocities $V$ (or momentums $P$) and give a cost at a specific time:
\begin{align*}
    \mathcal{U}_{*} (X) &= \frac{1}{2} \int_{\Omega^2} \Psi(X(\omega) - X(\tilde{\omega})) \, d\omega d\tilde{\omega} \\
    \hat{L}(X,V) &= \int_{\Omega} L(X(\omega), V(\omega)) \, d\omega + \frac{1}{2} \int_{\Omega^2} \Phi(X(\omega) - X(\tilde{\omega})) \, d\omega d\tilde{\omega} \\
    \hat{H}(X,P) &= \int_{\Omega} H(X(\omega), P(\omega)) \, d\omega - \frac{1}{2} \int_{\Omega^2} \Phi(X(\omega) - X(\tilde{\omega})) \, d\omega d\tilde{\omega}
\end{align*}

In other words,
\begin{equation*}
    \hat{L}(X,V) = \underset{\text{Running Cost over all players $\omega \in \Omega$}}{\underbrace{\int_{\Omega} L(X(\omega), V(\omega)) \, d\omega}} + \underset{\text{Interaction between any two players $\omega$, $\tilde{\omega} \in \Omega$}}{\underbrace{\frac{1}{2}\int_{\Omega^2} \Phi(X(\omega) - X(\tilde{\omega})) \, d\omega d\tilde{\omega}}}
\end{equation*}

And thus, we can consider the action $\hat{\mathcal{A}}_0^{T} : AC_2 (0,T; \mathbb{H}) \to \mathbb{R}$, which gives us the total cost for all players across time $[0,T]$,
\begin{equation*}
    \hat{\mathcal{A}}_0^{T}(\gamma) = \int_0^T \hat{L}(\gamma(t, \cdot), \dot{\gamma}(t, \cdot)) \, dt
\end{equation*}

where $AC_2(0,T;\mathbb{H})$ is the set of all $\gamma : [0,T] \to \mathbb{H}$ such that
\begin{equation*}
    \int_0^T \| \gamma(t, \cdot) \|_{\mathbb{H}}^2 \, dt < \infty
\end{equation*}

in which $\| \gamma(t, \cdot) \|_{\mathbb{H}} = \sqrt{\int_{\Omega} |\gamma(t, \cdot)|^2 \, d\omega}$.
\section{A Unique Nash Equilibrium}

For some player $q \in \mathbb{R}^d$, let $A$ denote their particular control (velocity), $\alpha^A$ denote the collective controls of all players (given that $A$ is the control of player $q$), $X^{\alpha^{A}}$ denote the collective trajectories of the collective controls such that $X_*$ denotes the final position (the official ranking of players), and let $q = X_*(\omega_0)$ for $\omega_0 \in \Omega$. We calculate the \emph{cost} of player $q$ via $J$ defined below:
\begin{multline*}
    J[T,q,A,\alpha^A] = \int_0^T L\left(X^{\alpha^{A}}(t, \omega_0), \alpha^{A} (t,\omega_0) \right) \, dt + \int_0^T\int_{\Omega} \Phi\left( X^{\alpha^{A}}(t, \omega_0) - X^{\alpha^{A}}(t, \omega) \right) \, d\omega dt \\
    + \int_{\Omega} \Psi\left( q - X^{\alpha^{A}}(0,\omega) \right) \, d\omega
\end{multline*}

We say that a collective control $\alpha$ is a \emph{Nash Equilibrium} if
\begin{equation*}
    J[T,q,\alpha(\cdot, \omega_0), \alpha] \leq J[T,q, A, \alpha] \quad \forall q \in \mathbb{R}^d, \forall A : [0,T] \to \mathbb{R}^d \; \mathcal{L}^1\text{--measurable}
\end{equation*}

That is, if any player $q$ switches from control $\alpha$ to $A$, it increases their cost. Our goal is to effectively show that given a Mean Field Game satisfying our initial assumptions, we can guarantee the existence and uniqueness of a Nash Equilibrium. To do so we utilize the following theorem:

\begin{theorem}
    \label{primary_thm}
    Suppose $u : [0,T] \times \mathbb{R}^d \to \mathbb{R}$, $X: [0,T] \times \Omega \to \mathbb{R}^d$ are of class $C^1$, $\nabla_{qq}u$ is bounded and
    \begin{equation*}
        \begin{dcases}
            \partial_t u(t,q) + H(q, \nabla_q u(t,q)) - \int_{\Omega} \Phi(q - X(t,\omega)) \, d\omega = 0 \\
            \dot{X}(t,\omega) = \nabla_p H(X(t,\omega), \nabla_q u(t, X(t,\omega))) \\
            u(0,q) = \int_{\Omega} \Psi(q - X(0, \omega)) \, d\omega \\
            X(T, \omega) = X_{*}(\omega) \quad (X_{*} \in \mathbb{H} \text{ is prescribed})
        \end{dcases}
    \end{equation*}
    
    If $X_{*}$ is non-atomic, then the game admits a unique Nash equilibrium.
\end{theorem}

\begin{proof}
    Fix $q \in \mathbb{R}^d, \omega_0 \in \Omega$ such that $q = X_{*}(\omega_0)$. Set
    \begin{equation*}
        \alpha(t,\omega) = \dot{X}(t, \omega) \quad \forall t \in (0,T), \forall \omega \in \Omega
    \end{equation*}
    
    Assume that player $q$ changes their control to $A: [0,T] \to \mathbb{R}^d$ while all other players keep their control unchanged. Set
    \begin{align*}
        \alpha^{A}(t,\omega) &= 
        \begin{dcases}
            A(t) & \text{if } X_{*}(\omega) = X_{*}(\omega_0) \\
            \alpha(t,\omega) & \text{if } X_{*}(\omega) \neq X_{*}(\omega_0)
        \end{dcases} \\
        \Omega_{\omega_0} &= \{\omega \in \Omega : X_{*}(\omega) = X_{*}(\omega_0)\}
    \end{align*}
    
    Note that
    \begin{equation*}
        \int_{\Omega} \Phi\left( q - X^{\alpha^{A}}(t, \omega) \right) \, d\omega = \int_{\Omega_{\omega_0}} \Phi\left( q - X^{\alpha^{A}}(t, \omega) \right) \, d\omega + \int_{\Omega_{\omega_0}^{C}} \Phi\left( q - X^{\alpha^{A}}(t, \omega) \right) \, d\omega
    \end{equation*}
    
    As the first integral is over a set of null measure ($\mathcal{L}^{1}(\Omega_{\omega_0}) = 0$), the integral equates to 0 and therefore
    \begin{align*}
        \int_{\Omega} \Phi\left( q - X^{\alpha^{A}}(t, \omega) \right) \, d\omega &= \int_{\Omega_{\omega_0}^{C}} \Phi\left( q - X^{\alpha^{A}}(t, \omega) \right) \, d\omega \\
        &= \int_{\Omega} \Phi\left( q - X^{\alpha}(t, \omega) \right) \, d\omega
    \end{align*}
    
    \begin{claim}
        \label{primary_thm_claim_1}
        Let $I(t) = u\left(t, X^{\alpha^{A}}(t, \omega_0)\right) - J[t,q,A,\alpha^{A}]$. We claim that
        \begin{equation*}
            \frac{d}{dt} I(t) < \frac{d}{dt} \left[ u(t,X^{\alpha}(t,\omega_0)) - J[t,q, \alpha(\cdot,\omega_0), \alpha] \right]
        \end{equation*}
        
        unless $A = \alpha(\cdot,\omega_0)$.
    \end{claim}
    \begin{proof}[Proof of Claim \ref{primary_thm_claim_1}]
        Directly computing,
        \begin{align*}
            \dot{I}(t) &= \partial_t u\left(t, X^{\alpha^{A}}(t, \omega_0)\right) + \left\langle \nabla_q u\left(t,X^{\alpha^{A}}(t, \omega_0)\right), \dot{X}^{\alpha^{A}}(t,\omega_0) \right\rangle\\ 
            &- L\left(X^{\alpha^{A}}(t,\omega_0), \dot{X}^{\alpha^{A}}(t,\omega_0)\right) - \int_{\Omega} \Phi\left( X^{\alpha^{A}}(t, \omega_0) - X^{\alpha^{A}}(t, \omega) \right) \, d\omega \\
            &= -H\left(X^{\alpha^{A}}(t,\omega_0), \nabla_q u\left(t, X^{\alpha^{A}}(t,\omega_0)\right)\right) + \int_{\Omega} \Phi \left( X^{\alpha^{A}}(t\omega_0) - X^{\alpha}(t,\omega) \right) \, d\omega \\
            &+ \left\langle \nabla_q u\left(t,X^{\alpha^{A}}(t, \omega_0)\right), \dot{X}^{\alpha^{A}}(t,\omega_0) \right\rangle - L\left(X^{\alpha^{A}}(t,\omega_0), \dot{X}^{\alpha^{A}}(t,\omega_0)\right) \\
            &- \int_{\Omega} \Phi\left( X^{\alpha^{A}}(t, \omega_0) - X^{\alpha}(t, \omega) \right) \, d\omega \\
            &\leq 0
        \end{align*}
        
        We have that the inequality is strict unless
        \begin{equation*}
            \begin{dcases}
                    \dot{X}^{\alpha^{A}}(t,\omega_0) = \nabla_q u\left(t, X^{\alpha^{A}} (t,\omega_0)  \right) \\
                    X^{\alpha^{A}}(T,\omega_0) = X_{*}(\omega_0)
            \end{dcases}
        \end{equation*}
        
        This implies  that $X^{\alpha^{A}}(t, \omega_0) = X^{\alpha}(t,\omega_0) \; \forall t \in [0,T] \implies A(t) = \dot{X}(t, \omega_0)$. This proves the claim.
    \end{proof}
    
    \begin{claim}
        \label{primary_thm_claim_2}
        Suppose $\nabla_{qq} u(t, \cdot)$ is bounded. We claim that
        \begin{equation*}
            \begin{dcases}
                    \dot{Y}(t) = \nabla_q H(Y(t), \nabla_q u(t, Y(t))) \\
                    Y(t) = X_{*}
            \end{dcases}
        \end{equation*}
        
        admits a unique solution.
    \end{claim}
    \begin{proof}[Proof of Claim \ref{primary_thm_claim_2}]
        Set $F(t,y) = \nabla_p H(y, \nabla_q u(t,y))$. Then
        \begin{align*}
            F(t,y_1) - F(t, y_0) &= \int_0^1 \frac{d}{ds} F(t, y_0 + s(y_1 - y_0)) \, ds \\ 
            &= \int_0^1 \langle \nabla_y F(t, y_0 + s(y_1 - y_0)), y_1 - y_0 \rangle \, ds
        \end{align*}
        
        Via the Chain Rule,
        \begin{equation*}
            \nabla_y F(t,y) = \nabla_{qp} H(y, \nabla_q u(t,y)) + \nabla_{pp}H(y,\nabla_q u(t,y)) \nabla_{qq} u(t,y)
        \end{equation*}
        
        And as $\nabla_{qq} u(t, \cdot)$ is bounded, we conclude $\nabla_y F(t,y)$ is bounded, which (it is easy to show that this) implies $F$ is Lipschitz. Hence to solve the system, we must solve 
        \begin{equation*}
            y(t) = X_{*} - \int_t^T F(s, y(s)) \, ds
        \end{equation*}
        
        Using Picard's Method, we define inductively 
        \begin{align*}
            y^{n + 1}(t) &= X_{*} - \int_t^T F(s, y^{n}(s)) \, ds \\
            y^{0}(s) &= X_* \quad \forall s
        \end{align*}
        
        It follows that $F(s, \cdot)$ Lipschitz $\implies (y^n)_n$ converges to a unique limit.
    \end{proof}
    
    Thus, integration on Claim \ref{primary_thm_claim_1} yields
    \begin{multline*}
        u\left( T, X^{\alpha^{A}}(T, \omega_0) \right) - J\left[T,q,A,\alpha^{A}\right] - u\left( 0, X^{\alpha^{A}}(0, \omega_0) \right) + J\left[0,q,A,\alpha^{A}\right] \\
        < u(T, X^{\alpha}(T, \omega_0)) - J\left[T,q, \alpha(\cdot,\omega_0), \alpha\right] - u(0, X^{\alpha}(0,\omega_0)) + J\left[0,q,A,\alpha^{A}\right]
    \end{multline*}
    
    unless $A = \alpha(\cdot,\omega_0)$. Cancellation gives us
    \begin{align*}
        -J\left[T,q,A,\alpha^{A}\right] &< -J[T,q,\alpha(\cdot,\omega_0), \alpha] \\
        J[T,q,\alpha(\cdot,\omega_0), \alpha] &< J\left[T,q,A,\alpha^{A}\right]
    \end{align*}
    
    unless $A = \alpha(\cdot,\omega_0)$, which gives us the existence of the Nash Equilibrium. Immediately from Claim \ref{primary_thm_claim_2} we conclude that the Nash Equilibrium is indeed unique.
\end{proof}

Hence, if we can demonstrate that the conditions of the system hold, we can effectively conclude the existence and uniqueness of a Nash Equilibrium. For the following sections we proceed assuming that $X \in L^{2} ((0,1); \mathbb{R}^d) \cap C^{1} ([0,1]; \mathbb{R}^d)$ (the range of $X$ is $\mathbb{R}^d$), and that $X$ is non-atomic ($X^{-1}(q)$ is a set of null measure in $\Omega$ for any $q \in \mathbb{R}^d$). 
\section{The Existence and Uniqueness of Optimal Cost}

\subsection{Existence and Uniqueness}
Consider the following function below:
\begin{equation*}
    \mathcal{\hat{B}}_0^{T}(\gamma) = \mathcal{\hat{A}}_0^{T}(\gamma) + \mathcal{U}_{*}(\gamma(0))
\end{equation*}

$\mathcal{\hat{B}}_0^T$ effectively calculates the total cost across all (finite) time. Our goal is to minimize such a cost, or rather, to show that a minimizer exists, i.e. prove the existence of minimizer $\gamma$ to obtain an optimal cost function.

\begin{theorem}
    \label{exist_uhat_thm}
    Given $\hat{\mathcal{B}}_0^T$ defined above, we have that a minimizer exists: that is, we have a minimum $\hat{U}$ where:
    \begin{equation*}
        \hat{U}(T, X) = \inf_{\gamma} \left\{ \mathcal{\hat{A}}_0^{T}(\gamma) + \mathcal{U}_{*}(\gamma(0)) : \gamma \in AC_2(0,T; \mathbb{H}), \gamma(T) = X \right\} 
    \end{equation*}
    
    Where $\gamma = \gamma_T^t[X]$ denotes the minimizer of $\hat{\mathcal{B}}_0^T$.
\end{theorem}

\begin{proof}
    We proceed using the Direct Method in the Calculus of Variations. First, let $(\gamma_n) \in AC_2(0,T; \mathbb{H})$ be a minimizing sequence on $\mathcal{\hat{B}}_0^T$:
    \begin{equation*}
        {\hat{\mathcal{B}}}_0^{T}(\gamma_n) \to \inf_{\bar{\gamma}} \left\{  {\hat{\mathcal{B}}}_0^{T}(\bar{\gamma}) \right\} \overset{\text{set}}{=} c
    \end{equation*}
    
    By definition, we have
    \begin{multline*}
        c + 1 \geq {\hat{\mathcal{B}}}_0^{T}(\gamma_n) 
        = \int_0^T\int_{\Omega} L(\gamma_n, \dot{\gamma_n}) \, d\omega dt 
        + \frac{1}{2} \int_0^T\int_{\Omega^2} \Phi(\gamma_n(t,\omega) - \gamma_n(t, \tilde{\omega})) \, d\omega d\tilde{\omega} dt \\
        + \frac{1}{2} \int_{\Omega^2} \Psi(\gamma_n(0,\omega) - \gamma_n(0, \tilde{\omega})) \, d\omega d\tilde{\omega} \\
        \geq c_0\int_0^T\int_{\Omega} |\dot{\gamma_n}|^2 \, d\omega dt - K_1 - K_2
    \end{multline*}

    Via our initial assumptions (i.e., bounds on $L, \Phi, \Psi$). Now, via the Poincar\'{e} Inequality, there exists some constant $K > 0$ such that
    \begin{equation*}
        c_0\int_0^T\int_{\Omega} |\dot{\gamma_n}|^2 \, d\omega dt \geq K\int_0^T\int_{\Omega} |\gamma_n|^2 \, d\omega dt 
    \end{equation*}

    Giving us
    \begin{align*}
        K\int_0^T\int_{\Omega} |\gamma_n|^2 \, d\omega dt - K_1 - K_2 &\leq c + 1 \\
        \left\| \gamma_n \right\|_{L^2}^2 = \int_0^T\int_{\Omega} |\gamma_n|^2 \, d\omega dt &\leq \frac{1}{K}\left(c + K_1 + K_2 + 1\right) \overset{\text{set}}{=} C^2
    \end{align*}

    Note that $C^2 > 0$ does not depend on $n$. Square-rooting both sides yields:
    \begin{equation*}
    \left\| \gamma_n \right\|_{L^2} \leq C \quad \forall n
    \end{equation*}
    
    Thus, we obtain the boundedness of $(\gamma_n)$. Recall that in $L^2$, every bounded sequence is weakly pre-compact, and conclude that our minimizing sequence is weakly pre-compact: that is, it admits a weakly convergent subsequence $(\gamma_{n_k})$ to some $\gamma_0$:
    
    \begin{equation*}
        \lim_{k \to \infty} \langle \gamma_{n_k}, g \rangle = \langle \gamma_0 , g \rangle \quad \forall g \in L^2
    \end{equation*}
    
    Next, we demonstrate the convexity of $\hat{\mathcal{B}}_0^T$. Consider some $\gamma, \bar{\gamma} \in L^2((0,T) \times \Omega; \mathbb{R}^d)$ where $\gamma(T, \cdot) = \bar{\gamma}(T, \cdot)$. Consider interpolation $\gamma^{\varepsilon} = (1 - \varepsilon)\gamma + \varepsilon\bar{\gamma}$ and set $f(\varepsilon) = {\hat{\mathcal{B}}}_0^{T}(\gamma^{\varepsilon}) = {\hat{\mathcal{A}}}_0^{T}(\gamma^{\varepsilon}) + \mathcal{U}_{*}(\gamma^{\varepsilon}(0, \cdot))$:
    \begin{multline*}
        f(\varepsilon) = \int_0^T\int_{\Omega} L(\gamma^{\varepsilon}, \dot{\gamma^{\varepsilon}}) \, d\omega dt 
        + \frac{1}{2} \int_0^T\int_{\Omega^2} \Phi(\gamma^{\varepsilon}(t,\omega) - \gamma^{\varepsilon}(t, \tilde{\omega})) \, d\omega d\tilde{\omega} dt \\
        + \frac{1}{2} \int_{\Omega^2} \Psi(\gamma^{\varepsilon}(0,\omega) - \gamma^{\varepsilon}(0, \tilde{\omega})) \, d\omega d\tilde{\omega}
    \end{multline*}

    Differentiating $f$ we obtain 
    \begin{multline*}
        f'(\varepsilon) = \int_0^T\int_{\Omega} \langle \nabla_{q} L(\gamma^{\varepsilon}, \dot{\gamma^{\varepsilon}}), \bar{\gamma} - \gamma \rangle + \langle \nabla_{v} L(\gamma^{\varepsilon}, \dot{\gamma^{\varepsilon}}), \dot{\bar{\gamma}} - \dot{\gamma} \rangle \, d\omega dt \\
        + \frac{1}{2}\int_0^T\int_{\Omega^2} \langle \nabla\Phi(\gamma^{\varepsilon}(t,\omega) - \gamma^{\varepsilon}(t,\tilde{\omega})), (\bar{\gamma}(t, \omega) - \gamma(t, \omega)) - (\bar{\gamma}(t, \tilde{\omega}) - \gamma(t, \tilde{\omega})) \rangle \, d\omega d\tilde{\omega} dt \\
        + \frac{1}{2}\int_{\Omega^2} \langle \nabla\Psi(\gamma^{\varepsilon}(0,\omega) - \gamma^{\varepsilon}(0,\tilde{\omega})), (\bar{\gamma}(0, \omega) - \gamma(0, \omega)) - (\bar{\gamma}(0, \tilde{\omega}) - \gamma(0, \tilde{\omega})) \rangle \, d\omega d\tilde{\omega}
    \end{multline*}

    If we consider the inner product containing $\Phi$, for example, see that we can rewrite it accordingly:
    \begin{multline*}
        \langle \nabla\Phi(\gamma^{\varepsilon}(t,\omega) - \gamma^{\varepsilon}(t,\tilde{\omega})), (\bar{\gamma}(t, \omega) - \gamma(t, \omega)) - (\bar{\gamma}(t, \tilde{\omega}) - \gamma(t, \tilde{\omega})) \rangle \\ 
        = \langle \nabla\Phi(\gamma^{\varepsilon}(t,\omega) - \gamma^{\varepsilon}(t,\tilde{\omega})), \bar{\gamma}(t, \omega) \gamma(t, \omega) \rangle \\ 
        - \langle \nabla\Phi(\gamma^{\varepsilon}(t,\omega) - \gamma^{\varepsilon}(t,\tilde{\omega})), \bar{\gamma}(t, \tilde{\omega}) - \gamma(t, \tilde{\omega}) \rangle
    \end{multline*}

    If we take the integral on the right hand side, we may split it up into two integrals: $\frac{1}{2}\int_0^T\int_{\Omega^2} \langle \nabla\Phi(\gamma^{\varepsilon}),\\ \bar{\gamma}(t,\omega) - \gamma(t, \omega) \rangle - \frac{1}{2}\int_0^T\int_{\Omega^2} \langle \nabla\Phi(\gamma^{\varepsilon}), \bar{\gamma}(t,\tilde{\omega}) - \gamma(t, \tilde{\omega}) \rangle$. We note that, in the second integral,
    \begin{multline*}
        -\frac{1}{2}\int_0^T\int_{\Omega^2} \langle \nabla\Phi(\gamma^{\varepsilon}(t,\omega) - \gamma^{\varepsilon}(t,\tilde{\omega})), \bar{\gamma}(t, \tilde{\omega}) - \gamma(t, \tilde{\omega}) \rangle \, d\omega d\tilde{\omega} dt\\ 
        = -\frac{1}{2}\int_0^T\int_{\Omega^2} \langle \nabla\Phi(\gamma^{\varepsilon}(t,\tilde{\omega}) - \gamma^{\varepsilon}(t,\omega)), \bar{\gamma}(t, \omega) - \gamma(t, \omega) \rangle \, d\omega d\tilde{\omega} dt \\
        = \frac{1}{2}\int_0^T\int_{\Omega^2} \langle -\nabla\Phi(\gamma^{\varepsilon}(t,\tilde{\omega}) - \gamma^{\varepsilon}(t,\omega)), \bar{\gamma}(t, \omega) - \gamma(t, \omega) \rangle \, d\omega d\tilde{\omega} dt \\
        = \frac{1}{2}\int_0^T\int_{\Omega^2} \langle \nabla\Phi(\gamma^{\varepsilon}(t,\omega) - \gamma^{\varepsilon}(t,\tilde{\omega})), \bar{\gamma}(t, \omega) - \gamma(t, \omega) \rangle \, d\omega d\tilde{\omega} dt
    \end{multline*}

    Effectively giving us another copy of the ``first" integral. With this reasoning we simplify
    \begin{multline*}
        \frac{1}{2}\int_0^T\int_{\Omega^2} \langle \nabla\Phi(\gamma^{\varepsilon}(t,\omega) - \gamma^{\varepsilon}(t,\tilde{\omega})), (\bar{\gamma}(t, \omega) - \gamma(t, \omega)) - (\bar{\gamma}(t, \tilde{\omega}) - \gamma(t, \tilde{\omega})) \rangle \, d\omega d\tilde{\omega} dt \\
        = \int_0^T\int_{\Omega^2} \langle \nabla\Phi(\gamma^{\varepsilon}(t,\omega) - \gamma^{\varepsilon}(t,\tilde{\omega})), \bar{\gamma}(t, \omega) - \gamma(t, \omega) \rangle \, d\omega d\tilde{\omega} dt
    \end{multline*}

    \newpage
    And
    \begin{multline*}
        \frac{1}{2}\int_{\Omega^2} \langle \nabla\Psi(\gamma^{\varepsilon}(0,\omega) - \gamma^{\varepsilon}(0,\tilde{\omega})), (\bar{\gamma}(0, \omega) - \gamma(0, \omega)) - (\bar{\gamma}(0, \tilde{\omega}) - \gamma(0, \tilde{\omega})) \rangle \, d\omega d\tilde{\omega} \\
        = \int_{\Omega^2} \langle \nabla\Psi(\gamma^{\varepsilon}(0,\omega) - \gamma^{\varepsilon}(0,\tilde{\omega})), \bar{\gamma}(0, \omega) - \gamma(0, \omega) \rangle \, d\omega d\tilde{\omega}
    \end{multline*}

    Giving us a slightly nicer expression for $f'(\varepsilon)$
    \begin{multline*}
        f'(\varepsilon) = \int_0^T\int_{\Omega} \langle \nabla_{q} L(\gamma^{\varepsilon}, \dot{\gamma^{\varepsilon}}), \bar{\gamma} - \gamma \rangle + \langle \nabla_{v} L(\gamma^{\varepsilon}, \dot{\gamma^{\varepsilon}}), \dot{\bar{\gamma}} - \dot{\gamma} \rangle \, d\omega dt \\
        + \int_0^T\int_{\Omega^2} \langle \nabla\Phi(\gamma^{\varepsilon}(t,\omega) - \gamma^{\varepsilon}(t,\tilde{\omega})), \bar{\gamma}(t, \omega) - \gamma(t, \omega) \rangle \, d\omega d\tilde{\omega} dt \\
        + \int_{\Omega^2} \langle \nabla\Psi(\gamma^{\varepsilon}(0,\omega) - \gamma^{\varepsilon}(0,\tilde{\omega})), \bar{\gamma}(0, \omega) - \gamma(0, \omega) \rangle \, d\omega d\tilde{\omega}
    \end{multline*}

    Differentiating again, we obtain $f''(\varepsilon)$:
    \begin{multline*}
        f''(\varepsilon) 
        = \int_0^T\int_{\Omega} \left\langle \nabla^2 L(\gamma^{\varepsilon}, \dot{\gamma^{\varepsilon}}) \begin{bmatrix} \bar{\gamma} - \gamma \\ \dot{\bar{\gamma}} - \dot{\gamma} \end{bmatrix}, \begin{bmatrix} \bar{\gamma} - \gamma \\ \dot{\bar{\gamma}} - \dot{\gamma} \end{bmatrix} \right\rangle \, d\omega dt \\
        + \int_0^T\int_{\Omega^2} \langle \nabla^2 \Phi(\gamma^{\varepsilon}(t,\omega) - \gamma^{\varepsilon}(t,\tilde{\omega}))\left( (\bar{\gamma}(t,\omega) - \gamma(t,\omega)) - (\bar{\gamma}(t, \tilde{\omega}) - \gamma(t, \tilde{\omega})) \right) , \bar{\gamma}(t,\omega) - \gamma(t,\omega) \rangle \, d\omega d\tilde{\omega} dt \\
        + \int_{\Omega^2} \langle \nabla^2 \Psi(\gamma^{\varepsilon}(0,\omega) - \gamma^{\varepsilon}(0,\tilde{\omega}))\left( (\bar{\gamma}(0,\omega) - \gamma(0,\omega)) - (\bar{\gamma}(0, \tilde{\omega}) - \gamma(0, \tilde{\omega})) \right) , \bar{\gamma}(0,\omega) - \gamma(0,\omega) \rangle \, d\omega d\tilde{\omega}
    \end{multline*}
    
    Splitting up the inner products gives us
    \begin{align*}
        f''(\varepsilon)
        &= \int_0^T\int_{\Omega} \left\langle \nabla^2 L(\gamma^{\varepsilon}, \dot{\gamma^{\varepsilon}}) \begin{bmatrix} \bar{\gamma} - \gamma \\ \dot{\bar{\gamma}} - \dot{\gamma} \end{bmatrix}, \begin{bmatrix} \bar{\gamma} - \gamma \\ \dot{\bar{\gamma}} - \dot{\gamma} \end{bmatrix} \right\rangle \, d\omega dt \\
        &+ \int_0^T\int_{\Omega^2} \langle \nabla^2 \Phi(\gamma^{\varepsilon}(t,\omega) - \gamma^{\varepsilon}(t,\tilde{\omega}))(\bar{\gamma}(t,\omega) - \gamma(t,\omega)) , \bar{\gamma}(t,\omega) - \gamma(t,\omega) \rangle \, d\omega d\tilde{\omega} dt \\
        &- \int_0^T\int_{\Omega^2} \langle \nabla^2 \Phi(\gamma^{\varepsilon}(t,\omega) - \gamma^{\varepsilon}(t,\tilde{\omega}))(\bar{\gamma}(t,\tilde{\omega}) - \gamma(t,\tilde{\omega})) , \bar{\gamma}(t,\omega) - \gamma(t,\omega) \rangle \, d\omega d\tilde{\omega} dt \\
        &+ \int_{\Omega^2} \langle \nabla^2 \Psi(\gamma^{\varepsilon}(0,\omega) - \gamma^{\varepsilon}(0,\tilde{\omega}))(\bar{\gamma}(0,\omega) - \gamma(0,\omega)) , \bar{\gamma}(0,\omega) - \gamma(0,\omega) \rangle \, d\omega d\tilde{\omega} \\
        &- \int_{\Omega^2} \langle \nabla^2 \Psi(\gamma^{\varepsilon}(0,\omega) - \gamma^{\varepsilon}(0,\tilde{\omega}))(\bar{\gamma}(0,\tilde{\omega}) - \gamma(0,\tilde{\omega})) , \bar{\gamma}(0,\omega) - \gamma(0,\omega) \rangle \, d\omega d\tilde{\omega}
    \end{align*}

    From which, through our initial assumptions,
    \begin{multline*}
        f''(\varepsilon) \geq 
        \int_0^T\int_{\Omega} \lambda_0|\dot{\bar{\gamma}} - \dot{\gamma}|^2 \, d\omega dt
        + \int_0^T\int_{\Omega^2}  c_{\Phi}|\bar{\gamma}(t,\omega) - \gamma(t,\omega)|^2 \, d\omega d\tilde{\omega} dt \\
        + \int_{\Omega^2} c_{\Psi}|\bar{\gamma}(0,\omega) - \gamma(0,\omega)|^2 \, d\omega d\tilde{\omega} \\
        - \int_0^T\int_{\Omega^2} \langle \nabla^2 \Phi(\gamma^{\varepsilon}(t,\omega) - \gamma^{\varepsilon}(t,\tilde{\omega}))(\bar{\gamma}(t,\tilde{\omega}) - \gamma(t,\tilde{\omega})) , \bar{\gamma}(t,\omega) - \gamma(t,\omega) \rangle \, d\omega d\tilde{\omega} dt \\
        - \int_{\Omega^2} \langle \nabla^2 \Psi(\gamma^{\varepsilon}(0,\omega) - \gamma^{\varepsilon}(0,\tilde{\omega}))(\bar{\gamma}(0,\tilde{\omega}) - \gamma(0,\tilde{\omega})) , \bar{\gamma}(0,\omega) - \gamma(0,\omega) \rangle \, d\omega d\tilde{\omega}
    \end{multline*} 

    We utilize the following fact:
    \begin{equation*}
        \langle \nabla^2 f(x)a, b \rangle \leq \text{const}(\nabla^2 f(x)) (|a|^2 + |b|^2)
    \end{equation*}

    Let $C_1 = \alpha(\nabla^2 \Phi), C_2 = \beta(\nabla^2 \Psi)$ be such that $\langle \nabla^2 \Phi(\cdot) a, b \rangle \leq C_1(|a|^2 + |b|^2)$ and $\langle \nabla^2 \Psi(\cdot) a, b \rangle \leq C_2(|a|^2 + |b|^2)$. We make the next simplification:
    \begin{multline*}
        -\int_0^T\int_{\Omega^2} \langle \nabla^2 \Phi(\gamma^{\varepsilon}(t,\omega) - \gamma^{\varepsilon}(t,\tilde{\omega}))(\bar{\gamma}(t,\tilde{\omega}) - \gamma(t,\tilde{\omega})) , \bar{\gamma}(t,\omega) - \gamma(t,\omega) \rangle \, d\omega d\tilde{\omega} dt \\
        \geq -\int_0^T\int_{\Omega^2} C_1 \left(|\bar{\gamma}(t,\tilde{\omega}) - \gamma(t,\tilde{\omega})|^2 + |\bar{\gamma}(t,\omega) - \gamma(t,\omega)|^2 \right) \, d\omega d\tilde{\omega} dt \\
        = -C_1\biggl[ \biggr.\int_0^T\int_{\Omega^2} |\bar{\gamma}(t,\tilde{\omega}) - \gamma(t,\tilde{\omega})|^2  \, d\omega d\tilde{\omega} dt \\
        + \int_0^T\int_{\Omega^2} |\bar{\gamma}(t,\omega) - \gamma(t,\omega)|^2 \, d\omega d\tilde{\omega} dt \biggl. \biggr]
    \end{multline*}

    And we note that
    \begin{equation*}
        \int_0^T\int_{\Omega^2} |\bar{\gamma}(t,\tilde{\omega}) - \gamma(t,\tilde{\omega})|^2  \, d\omega d\tilde{\omega} dt = \int_0^T\int_{\Omega^2} |\bar{\gamma}(t,\omega) - \gamma(t,\omega)|^2 \, d\omega d\tilde{\omega} dt
    \end{equation*}

    Giving us
    \begin{multline*}
        -\int_0^T\int_{\Omega^2} \langle \nabla^2 \Phi(\gamma^{\varepsilon}(t,\omega) - \gamma^{\varepsilon}(t,\tilde{\omega}))(\bar{\gamma}(t,\tilde{\omega}) - \gamma(t,\tilde{\omega})) , \bar{\gamma}(t,\omega) - \gamma(t,\omega) \rangle \, d\omega d\tilde{\omega} dt \\
        \geq -\int_0^T\int_{\Omega^2} 2C_1 |\bar{\gamma}(t,\omega) - \gamma(t,\omega)|^2 \, d\omega d\tilde{\omega} dt
    \end{multline*}

    And, with the same reasoning,
    \begin{multline*}
        -\int_{\Omega^2} \langle \nabla^2 \Psi(\gamma^{\varepsilon}(0,\omega) - \gamma^{\varepsilon}(0,\tilde{\omega}))(\bar{\gamma}(0,\tilde{\omega}) - \gamma(0,\tilde{\omega})) , \bar{\gamma}(0,\omega) - \gamma(0,\omega) \rangle \, d\omega d\tilde{\omega} \\
        \geq -\int_{\Omega^2} 2C_2 |\bar{\gamma}(0,\omega) - \gamma(0,\omega)|^2 \, d\omega d\tilde{\omega}
    \end{multline*}

    Which allows us to rewrite our inequality for $f''(\varepsilon)$:
    \begin{multline*}
        f''(\varepsilon) \geq 
        \int_0^T\int_{\Omega} \lambda_0|\dot{\bar{\gamma}} - \dot{\gamma}|^2 \, d\omega dt
        + \int_0^T\int_{\Omega^2}  (c_{\Phi}^{+} - c_{\Phi}^{-} - 2C_1)|\bar{\gamma}(t,\omega) - \gamma(t,\omega)|^2 \, d\omega d\tilde{\omega} dt \\
        + \int_{\Omega^2} (c_{\Psi}^{+} - c_{\Psi}^{-} - 2C_2)|\bar{\gamma}(0,\omega) - \gamma(0,\omega)|^2 \, d\omega d\tilde{\omega}
    \end{multline*}

    Where we rewrite $c_{\Phi} = c_{\Phi}^{+} - c_{\Phi}^{-}$, $c_{\Psi} = c_{\Psi}^{+} - c_{\Psi}^{-}$. Notice now that the second and third integrals are constants with respect to $\tilde{\omega}$, and so we can evaluate directly and obtain
    \begin{multline*}
        f''(\varepsilon) \geq 
        \int_0^T\int_{\Omega} \lambda_0|\dot{\bar{\gamma}} - \dot{\gamma}|^2 \, d\omega dt
        + \int_0^T\int_{\Omega}  (c_{\Phi}^{+} - c_{\Phi}^{-} - 2C_1)|\bar{\gamma}(t,\omega) - \gamma(t,\omega)|^2 \, d\omega dt \\
        + \int_{\Omega} (c_{\Psi}^{+} - c_{\Psi}^{-} - 2C_2)|\bar{\gamma}(0,\omega) - \gamma(0,\omega)|^2 \, d\omega
    \end{multline*}

    We rewrite the first integral as follows:
    \begin{equation*}
        \int_0^T\int_{\Omega} \lambda_0|\dot{\bar{\gamma}} - \dot{\gamma}|^2 \, d\omega dt = 2\left(\int_0^T\int_{\Omega} \frac{\lambda_0}{2}|\dot{\bar{\gamma}} - \dot{\gamma}|^2 \, d\omega dt\right)
    \end{equation*}

    And we group accordingly:
    \begin{multline*}
        f''(\varepsilon) \geq \left[\int_0^T\int_{\Omega}  (c_{\Phi}^{+} - c_{\Phi}^{-} - 2C_1)|\bar{\gamma}(t,\omega) - \gamma(t,\omega)|^2 \, d\omega dt + \int_0^T\int_{\Omega} \frac{\lambda_0}{2}|\dot{\bar{\gamma}} - \dot{\gamma}|^2 \, d\omega dt \right] \\ 
        + \left[ \int_{\Omega} (c_{\Psi}^{+} - c_{\Psi}^{-} - 2C_2)|\bar{\gamma}(0,\omega) - \gamma(0,\omega)|^2 \, d\omega + \int_0^T\int_{\Omega} \frac{\lambda_0}{2}|\dot{\bar{\gamma}} - \dot{\gamma}|^2 \, d\omega dt \right] 
    \end{multline*}

    By the Poincar\'{e} Inequality, we have
    \begin{align*}
        \frac{\lambda_0}{2} \int_0^T |\dot{\bar{\gamma}}(t,\omega) - \dot{\gamma}(t,\omega)|^2 \, dt &\geq \frac{\lambda_0}{2} \cdot \frac{2}{T^2} \int_0^T |\bar{\gamma}(t,\omega) - \gamma(t,\omega)|^2 \, dt \\
        \frac{\lambda_0}{2} \int_0^T \int_{\Omega} |\dot{\bar{\gamma}}(t,\omega) - \dot{\gamma}(t,\omega)|^2 \, d\omega dt &\geq \frac{\lambda_0}{T^2} \int_0^T \int_{\Omega} |\bar{\gamma}(t,\omega) - \gamma(t,\omega)|^2 \, d\omega dt
    \end{align*}

    And also, 
    \begin{align*}
        \frac{\lambda_0}{2} \int_0^T |\dot{\bar{\gamma}}(t,\omega) - \dot{\gamma}(t,\omega)|^2 \, dt &\geq \frac{\lambda_0}{2} \cdot \frac{1}{T} |\bar{\gamma}(0,\omega) - \gamma(0,\omega)|^2 \\
        \frac{\lambda_0}{2} \int_0^T \int_{\Omega} |\dot{\bar{\gamma}}(t,\omega) - \dot{\gamma}(t,\omega)|^2 \, d\omega dt &\geq \frac{\lambda_0}{2T} \int_{\Omega} |\bar{\gamma}(0,\omega) - \gamma(0,\omega)|^2 \, d\omega
    \end{align*}

    Yielding
    \begin{multline*}
        f''(\varepsilon) \geq \int_0^T \int_{\Omega} \left(c_{\Phi}^{+} - c_{\Phi}^{-} - 2C_1 + \frac{\lambda_0}{T^2}\right)|\bar{\gamma}(t,\omega) - \gamma(t,\omega)|^2 \, d\omega dt \\ 
        + \int_{\Omega}\left( c_{\Psi}^{+} - c_{\Psi}^{-} - 2C_2 + \frac{\lambda_0}{2T} \right) |\bar{\gamma}(0,\omega) - \gamma(0,\omega)|^2 \, d\omega 
    \end{multline*}

    In which for small enough $T$, we have that both $c_{\Phi}^{+} - c_{\Phi}^{-} - 2C_1 + \frac{\lambda_0}{T^2}$ and $c_{\Psi}^{+} - c_{\Psi}^{-} - 2C_2 + \frac{\lambda_0}{2T}$ are nonnegative. Thus, we conclude that $f$ is convex, and it follows that as it is a slice of $\hat{\mathcal{B}}_0^T$, we conclude that $\hat{\mathcal{B}}_0^T$ is also convex. $\hat{\mathcal{B}}_0^T$ convex is more than sufficient to show that $\hat{\mathcal{B}}_0^T$ is weakly lower semi-continuous. That is, for any sequence $y_n \rightharpoonup y$,
    \begin{equation*}
        \liminf_{n \to \infty} \hat{\mathcal{B}}_0^T(y_n) \geq \hat{\mathcal{B}}_0^T(y)
    \end{equation*}
    
    And thus we conclude
    \begin{equation*}
        \inf_{\bar{\gamma}} \left\{ {\hat{\mathcal{B}}}_0^{T}(\bar{\gamma}) \right\} = \lim_{n \to \infty} {\hat{\mathcal{B}}}_0^{T}(\gamma_n) = \lim_{k \to \infty} {\hat{\mathcal{B}}}_0^{T}(\gamma_{n_k}) \geq \hat{\mathcal{B}}_0^T(\gamma_0) \geq \inf_{\bar{\gamma}} \left\{ {\hat{\mathcal{B}}}_0^{T}(\bar{\gamma}) \right\}
    \end{equation*}
    
    From which we conclude that a minimum exists for some minimizer $\gamma_0$:
    \begin{equation*}
        \hat{U}(T, X) = \inf_{\gamma} \left\{ \hat{\mathcal{B}}_0^T(\gamma) : \gamma \in AC_2(0,T; \mathbb{H}), \gamma(T) = X \right\}
    \end{equation*}

    We now show that such a minimizer is unique. Assume we have two distinct minimizers of $\hat{\mathcal{B}}_0^T$ $\gamma^0 \neq \gamma^1$. We aim to show that they are equal. We consider a convex combination of these minimizers $\gamma^\varepsilon = (1 - \varepsilon)\gamma^0 - \varepsilon \gamma^1$. Taylor's Theorem on $\hat{\mathcal{B}}_0^T$ yields 
    \begin{equation*}
        \hat{\mathcal{B}}_0^T(\gamma^1) - \hat{\mathcal{B}}_0^T(\gamma^0) = \frac{d}{d\varepsilon}  \left. \hat{\mathcal{B}}_0^T(\gamma^\varepsilon) \right|_{\varepsilon = 0} +  \frac{1}{2} \left. \frac{d^2}{d\varepsilon^2} \hat{\mathcal{B}}_0^T(\gamma^\varepsilon) \right|_{\varepsilon = \varepsilon^*}
    \end{equation*}
    
    For some $0 < \varepsilon^* < \varepsilon$. By definition, $\hat{\mathcal{B}}_0^T(\gamma^1) - \hat{\mathcal{B}}_0^T(\gamma^0) = 0$ as they are both minimizers. Additionally, as $\varepsilon = 0$ in $\gamma^\varepsilon$ yields a minimizer ($\gamma^0$), we conclude that $\left. \hat{\mathcal{B}}_0^T(\gamma^\varepsilon) \right|_{\varepsilon = 0} = 0$. Hence, from above we can also conclude that $\left. \frac{d^2}{d\varepsilon^2}\hat{\mathcal{B}}_0^T(\gamma^\varepsilon) \right|_{\varepsilon = \varepsilon^*} = 0$. From above we have that
    \begin{equation*}
        f''(\varepsilon) \geq  \left(\lambda_0 - \frac{T^2}{2}\left(c_{\Phi}^{-} + \| \nabla^2 \Phi \|_{L^\infty} \right) - T\left(c_{\Psi}^{-} + \| \nabla^2 \Psi \|_{L^\infty} \right) \right) \int_0^T \int_{\Omega} |\dot{\gamma}^1 - \dot{\gamma}^0|^2 \, d\omega dt \geq 0
    \end{equation*}
    
    Via our initial assumptions the above constant multiple is strictly positive, hence we conclude that if $f''(\varepsilon^*) = 0$ then
    \begin{equation*}
        \int_0^T \int_{\Omega} |\dot{\gamma}^1 - \dot{\gamma}^0|^2 \, d\omega dt = 0 \implies \dot{\gamma}^1 = \dot{\gamma}^0
    \end{equation*}
    
    As the terminal position for all minimizers is $X$ ($\gamma(T) = X$), we conclude that as $\gamma^1$ and $\gamma^0$ have equal derivatives and agree at a point, they must be equal everywhere: $\gamma^1 = \gamma^0$, and hence the minimizer is unique. We denote the unique minimizer of $\hat{U}(T,X)$ at time $t$ and given player label $\omega$ as $\gamma_T^t[X](\omega)$.
\end{proof}

\subsection{Euler--Lagrange Equation of Minimizer}

We now proceed to demonstrate that as a minimizer achieving $\hat{U}(t,X)$, $s \mapsto \gamma_t^s[X]$ is a solution of the Euler--Lagrange Equation:

\begin{proposition}
    \label{euler_lagrange_thm_uhat}
    $\gamma(s,\omega) = {\gamma}_t^s[X](\omega)$ is a minimizer achieving $\hat{U}(t,X)$ if and only if 
    \begin{equation*}
        \begin{dcases}
            \frac{d}{ds} \nabla_v L(\gamma(s, \omega), \dot{\gamma}(s, \omega)) = \nabla_q L(\gamma(s, \omega), \dot{\gamma}(s, \omega)) + \int_{\Omega} \nabla\Phi (\gamma(s, \omega) - \gamma(s, \tilde{\omega})) \, d\tilde{\omega} \\
            \gamma(t,\omega) = X(\omega) \\
            \int_{\Omega} \nabla \Psi(\gamma(0, \omega) - \gamma(0, \tilde{\omega})) \, d\tilde{\omega} = \nabla_v L(\gamma(0,\omega), \dot{\gamma}(0,\omega))
        \end{dcases}
    \end{equation*}
\end{proposition}

\begin{proof}
    Consider
    \begin{equation*}
        I(\bar{\gamma}) = \int_0^t \int_{\Omega} L(\bar{\gamma}, \dot{\bar{\gamma}}) + \left( \frac{1}{2} \int_{\Omega} \Phi(\gamma(s,\omega) - \gamma(s, \tilde{\omega})) \, d\tilde{\omega} \right) \, d\omega dt
    \end{equation*}
    
    Suppose $\bar{\gamma}(t) = X$. Suppose $\gamma(s, \omega) = \gamma_t^s[X](\omega)$ is the minimizer of $I$ ($I(\gamma) = \hat{U}(t,X)$). Consider a variation on $\gamma$: let $g$ be such that $g(0, \cdot) = g(t, \cdot) = g(\cdot, 0) = g(\cdot, 1) = 0$. Let $\varepsilon \in \mathbb{R}$. Then
    \begin{equation*}
        \begin{aligned}
            (\gamma + \varepsilon v)(0, \cdot) &= \gamma(0, \cdot) \\
            (\gamma + \varepsilon v)(t, \cdot) &= \gamma(t, \cdot)
        \end{aligned}
        \quad
        \text{and}
        \quad
        \begin{aligned}
            (\gamma + \varepsilon v)(\cdot, 0) &= \gamma(\cdot, 0) \\
            (\gamma + \varepsilon v)(\cdot, 1) &= \gamma(\cdot, 1)
        \end{aligned}
    \end{equation*}
    
    Let $a(\varepsilon) = I(\gamma + \varepsilon g)$. Then
    \begin{equation*}
        a(\varepsilon) = \int_0^t\int_{\Omega} L(\gamma + \varepsilon g, \dot{\gamma} + \varepsilon \dot{g}) + \left( \frac{1}{2} \int_{\Omega} \Phi((\gamma + \varepsilon g)(s,\omega) - (\gamma + \varepsilon g)(s, \tilde{\omega})) \, d\tilde{\omega} \right) \, d\omega dt
    \end{equation*}
    \begin{multline*}
        a'(\varepsilon) = \int_0^t\int_{\Omega} \langle \nabla_q L(\gamma + \varepsilon g, \dot{\gamma} + \varepsilon \dot{g}), g \rangle + \langle \nabla_v L(\gamma + \varepsilon g, \dot{\gamma} + \varepsilon \dot{g}), \dot{g} \rangle \\
        + \left(\int_{\Omega} \langle \nabla \Phi((\gamma + \varepsilon g)(s,\omega) - (\gamma + \varepsilon g)(s,\tilde{\omega})), g \rangle \, d\tilde{\omega} \right) \, d\omega dt
    \end{multline*}
    
    As $\gamma$ minimizes $I$, $\varepsilon = 0$ minimizes $a$, hence
    \begin{equation*}
        0 = a'(0) = \int_0^t\int_{\Omega} \langle \nabla_q L(\gamma, \dot{\gamma}), g \rangle + \langle \nabla_v L(\gamma, \dot{\gamma}), \dot{g} \rangle + \left\langle \int_{\Omega} \nabla \Phi(\gamma(s,\omega) - \gamma(s, \tilde{\omega})) \, d\tilde{\omega}, g \right\rangle \, d\omega dt
    \end{equation*}
    
    We rewrite via Integration by Parts:
    \begin{equation*}
        \int_0^t\int_{\Omega} \langle \nabla_v L(\gamma, \dot{\gamma}, g \rangle \, d\omega ds = \left. \int_{\Omega} \langle \nabla_v L(\gamma, \dot{\gamma}), g \rangle \, d\omega \right]_{0}^t - \int_0^t \int_{\Omega} \left\langle \frac{d}{ds} \nabla_v L(\gamma, \dot{\gamma}) , g \right\rangle \, d\omega ds
    \end{equation*}
    
    Note that
    \begin{equation*}
        \left. \int_{\Omega} \langle \nabla_v L(\gamma, \dot{\gamma}), g \rangle \, d\omega \right]_{0}^t = \int_{\Omega} \langle \nabla_v L(\gamma(t), \dot{\gamma}(t)), 0 \rangle \, d\omega - \int_{\Omega} \langle \nabla_v L(\gamma(0), \dot{\gamma}(0)), 0 \rangle \, d\omega = 0
    \end{equation*}
    
    Hence we obtain
    \begin{equation*}
        \int_0^t\int_{\Omega} \left\langle \nabla_q L(\gamma, \dot{\gamma}) - \frac{d}{ds} \nabla_v L(\gamma, \dot{\gamma}) + \int_{\Omega} \nabla \Phi(\gamma(s,\omega) - \gamma(s, \tilde{\omega})) \, d\tilde{\omega}, g \right\rangle \, d\omega ds = 0
    \end{equation*}
    
    As the above is true for all $g$, we conclude
    \begin{equation*}
        \frac{d}{ds} \nabla_v L(\gamma(s,\omega), \dot{\gamma}(s,\omega)) = \nabla_q L(\gamma(s,\omega), \dot{\gamma}(s,\omega)) + \int_{\Omega} \nabla \Phi(\gamma(s,\omega) - \gamma(s, \tilde{\omega})) \, d\tilde{\omega}
    \end{equation*}
    
    Where again, $\gamma(s,\omega) = \gamma_t^s[X](\omega)$. We show that
    \begin{equation*}
        \int_{\Omega} \nabla \Psi(\gamma(0,\omega) - \gamma(0,\tilde{\omega})) \, d\tilde{\omega} = \nabla_v L(\gamma(0, \omega), \dot{\gamma}(0, \omega))
    \end{equation*}
    
    By plugging in $t = 0$ in Claim \ref{uhat_claim_X_derivative} and thus obtain the desired result.
\end{proof}
\section{The Differentiability of \texorpdfstring{$\hat{U}$}{U-hat} and its Hamilton--Jacobi Equation}

We next aim to show that $\hat{U}$ is (continuously) differentiable (existence of $\partial_t\hat{U}, \nabla_X\hat{U} \in C$) and that it satisfies the Hamilton--Jacobi Equation
\begin{equation*}
    \partial_t \hat{U}(t,X) + \hat{H}(X, \nabla_X\hat{U}(t,X)) = 0
\end{equation*}

\subsection{Existence \& Continuity of \texorpdfstring{$\partial_t \hat{U}$}{Partial-t U-hat}, \texorpdfstring{$\nabla_X\hat{U}$}{Nabla-X U-hat}}

First recall the distributional derivative of some function $f$. If $f: \mathbb{R}^d \to \mathbb{R}$, let $\varphi : \mathbb{R}^d \to \mathbb{R}$ be of class $C_c^{\infty}(\mathbb{R}^d)$. We can define the distributional derivative $g$ as the function such that, for all $\varphi$,
\begin{equation*}
    \int_{\mathbb{R}^d} f \frac{\partial \varphi}{\partial x_i} \, dx = -\int_{\mathbb{R}^d} g \varphi \, dx 
\end{equation*}

Equivalently, we can look at higher power derivatives: let $g^{(n)}$ denote the $n$-th distributional derivative of $f$, then by definition,
\begin{equation*}
    \int_{\mathbb{R}^d} f \frac{\partial^n \varphi}{\partial x_{i_1} \partial x_{i_2} \cdots \partial x_{i_n}} \, dx = (-1)^n\int_{\mathbb{R}^d} g^{(n)} \varphi \, dx
\end{equation*}

When this is the case, we simply denote $g$ (and $g^{(n)}$) as
\begin{align*}
    g &= \frac{\partial f}{\partial x_i} \\
    g^{(n)} &= \frac{\partial^n f}{\partial x_{i_1} \partial x_{i_2} \cdots \partial x_{i_n}}
\end{align*}

A nice consequence about considering distributional derivatives is that they always exist, even when the classical derivative may not exist. Furthermore, if the classical derivative does exist, the distributional derivative is exactly equal to the classical derivative. Thus, we may use them in our discussion of demonstrating the differentiability of $\hat{U}$ (if the existence of the classical derivative in context is not certain, it may be assumed that it is the distributional derivative being used). 

Consider the following: for $f: \mathbb{H} \to \mathbb{R}$, if $-C \leq \nabla^2f \leq C$ (in a distributional sense), then $\nabla f$ exists and is $C$--Lipschitz (in a classical sense). If $\mathbb{H} = \mathbb{R}^d$, for example, $\nabla^2 f(x)$ is symmetric, and thus if $f$ is smooth enough,

\begin{equation*}
    \| \nabla^2f(x) \|^2 = \sum_{i,j = 1}^d \left| \frac{\partial^2 f}{\partial x_i \partial x_j} \right|^2 = \sum_{i = 1}^d \lambda_i^2(x)
\end{equation*}
    
Where $\{\lambda_i\}$ is the set of eigenvalues of $\nabla^2f(x)$. Let $C$ be such that $-C \leq \lambda_i \leq C \; \forall i$. Thus, $\| \nabla^2f(x) \|^2 \leq dC^2$. We can easily see then that, for $0 \leq t \leq 1$,
\begin{align*}
    \| \nabla f(x) - \nabla f(y) \| &= \left\| \int_0^1 \nabla^2 f((1 - t)x + ty)(x-y) \, dt \right\| \\
    &\leq \int_0^1 \| \nabla^2 f((1 - t)x + ty) \| \, dt \cdot \| x - y \| \\
    &\leq \sqrt{d}C \| x - y \|
\end{align*}

(If $\mathbb{H}$ is general, we have that $\dim(\text{span}\{x,y\}) = 2$, and thus $\| \nabla f(x) - \nabla f(y) \| \leq \sqrt{2}C \| x - y \|$). Thus, $\nabla f$ is $C$--Lipschitz. Note that, using Taylor's Theorem, 
\begin{align*}
    f(x + h) &= f(x) + \langle \nabla f(x), h \rangle + \int_0^1 \int_0^t \langle \nabla^2 f(x + sh)h , h \rangle \, ds dt \\
    f(x - h) &= f(x) - \langle \nabla f(x), h \rangle + \int_0^1 \int_0^t \langle \nabla^2 f(x - sh)h , h \rangle \, ds dt
\end{align*}

Thus
\begin{equation*}
    f(x + h) + f(x - h) - 2f(x) = \int_0^1 \int_0^t \langle \nabla^2 (f(x + sh) + f(x - sh))h , h \rangle \, ds dt
\end{equation*}

If $\|\nabla^2 f(x)\| \leq C$, we have then that
\begin{align*}
    |f(x + h) + f(x - h) - 2f(x)| &\leq \int_0^1 \int_0^t \langle 2Ch , h \rangle \, ds dt \\
    &= \int_0^1 2Ct \langle h,h \rangle \, dt \\
    &= C \langle h,h \rangle \\
    &= C|h|^2
\end{align*}

From which we can conclude
\begin{align*}
    \nabla^2 f(x) \geq -C &\iff f(x + h) + f(x - h) - 2f(x) \geq -C|h|^2 \\
    \nabla^2 f(x) \leq C &\iff f(x + h) + f(x - h) - 2f(x) \leq C|h|^2
\end{align*}

In other words, if we can show that $-C \leq \nabla_{XX} \hat{U}(t,X) \leq C$ for some constant $C$, then we can conclude that $\nabla_X \hat{U}(t,X)$ exists and that $\nabla_X \hat{U}(t,\cdot)$ is $C$--Lipschitz (hence, continuous). We use extra reasoning to conclude that $\nabla_X \hat{U}$ is also continuous on time.

\begin{theorem}
    \label{first_derivative_x_uhat}
    $\nabla_X \hat{U}(t,X)$ exists and is continuous.
\end{theorem}

\begin{proof}
    Denote
    \begin{align*}
        \bar{L}(X,V) &= \int_{\Omega} L(X(\omega), V(\omega)) \, d\omega \\
        \hat{\mathcal{F}}(X) &= \frac{1}{2} \int_{\Omega^2} \Phi(X(\omega) - X(\tilde{\omega})) \, d\omega d\tilde{\omega}
    \end{align*}

    And recall
    \begin{equation*}
        \hat{U}(t,X) = \inf_{\bar{\gamma}} \left\{ \int_0^t \bar{L}(\bar{\gamma}, \dot{\bar{\gamma}}) + \hat{\mathcal{F}}(\bar{\gamma}) \, d\tau + \mathcal{U}_{*}(\bar{\gamma}(0)) : \bar{\gamma}(t) = X \right\}
    \end{equation*} 
    
    Let $\gamma$ be optimal in $\hat{U}$ (in other words, $\bar{\gamma} = \gamma_t^\tau[X]$). That is,
    \begin{equation*}
        \hat{U}(t,X) = \int_0^t \bar{L}(\gamma, \dot{\gamma}) + \hat{\mathcal{F}}(\gamma) \, d\tau + \mathcal{U}_{*}(\gamma(0)) \quad\quad (\gamma(t) = X)
    \end{equation*}
    
    Now, we introduce a variation on $\gamma$: let $\tilde{\gamma}(\tau) = \gamma(\tau) + \frac{\tau}{t}h$. See that $\tilde{\gamma}(t) = \gamma(t) + h$, and $\tilde{\gamma}(0) = \gamma(0)$. Hence,
    \begin{align*}
        \hat{U}(t, X + h) &\leq \int_0^t \bar{L}\left(\gamma + \frac{\tau}{t}h, \dot{\gamma} + \frac{h}{t}\right) + \hat{\mathcal{F}}\left( \gamma + \frac{\tau}{t}h \right) \, d\tau + \mathcal{U}_{*} (\gamma(0)) \\
        \hat{U}(t, X - h) &\leq \int_0^t \bar{L}\left(\gamma - \frac{\tau}{t}h, \dot{\gamma} - \frac{h}{t}\right) + \hat{\mathcal{F}}\left( \gamma - \frac{\tau}{t}h \right) \, d\tau + \mathcal{U}_{*} (\gamma(0))
    \end{align*}
    
    Let $C_{\bar{L}}, C_{\hat{\mathcal{F}}}$ be such that $\nabla^2 \bar{L} \leq C_{\bar{L}}I$, $\nabla^2 \hat{\mathcal{F}} \leq C_{\hat{\mathcal{F}}}I$. Recall that
    \begin{equation*}
        f(x + h) + f(x - h) - 2f(x) \leq \max_{x} \{\langle \nabla^2 f(x) h, h \rangle \}
    \end{equation*}
    
    Combining these, we get
    \begin{align*}
        \hat{U}(t, X + h) + \hat{U}(t, X - h) - 2\hat{U}(t,X) &\leq \int_0^t C_{\bar{L}} \left\langle \begin{pmatrix} \frac{\tau h}{t} \\ \frac{h}{t} \end{pmatrix} , \begin{pmatrix} \frac{\tau h}{t} \\ \frac{h}{t} \end{pmatrix} \right\rangle + C_{\hat{\mathcal{F}}} \left\langle \frac{\tau h}{t}, \frac{\tau h}{t} \right\rangle \, d\tau \\
        &= \int_0^t C_{\bar{L}}\left( \frac{\tau^2 |h|^2}{t^2} + \frac{|h|^2}{t^2} \right) + C_{\hat{\mathcal{F}}}\left( \frac{\tau^2 |h|^2}{t^2} \right) \, d\tau \\
        &= \left. (C_{\bar{L}} + C_{\hat{\mathcal{F}}})\frac{\tau^3}{3t^2}|h|^2 + C_{\bar{L}}\frac{\tau}{t^2}|h|^2 \right]_{\tau = 0}^{\tau = t} \\
        &= \left( \frac{t}{3}(C_{\bar{L}} + C_{\hat{\mathcal{F}}}) + \frac{1}{t}C_{\bar{L}} \right)|h|^2
    \end{align*}
    
    Which implies that
    \begin{equation}
        \label{hard_case}
        \nabla_{XX} \hat{U}(t,X) \leq \frac{t}{3}(C_{\bar{L}} + C_{\hat{\mathcal{F}}}) + \frac{1}{t}C_{\bar{L}} \overset{\text{set}}{=} C_t
    \end{equation}
    
    For a lower bound on $\nabla_{XX} \hat{U}$, we simply consider optimal paths on its variations: let $\gamma^+$, $\gamma^-$ denote the optimal paths for $\hat{U}(t, X + h)$, $\hat{U}(t, X - h)$ respectively. Thus,
    \begin{align*}
        \hat{U}(t, X + h) &= \int_0^t \bar{L}(\gamma^+, \dot{\gamma}^+) + \hat{\mathcal{F}}(\gamma^+) \, d\tau + \mathcal{U}_*(\gamma^+(0)) \\
        \hat{U}(t, X - h) &= \int_0^t \bar{L}(\gamma^-, \dot{\gamma}^-) + \hat{\mathcal{F}}(\gamma^-) \, d\tau + \mathcal{U}_*(\gamma^-(0))
    \end{align*}
    
    Consider the path obtained by taking the average of the two optimal paths, $\frac{1}{2}(\gamma^+ + \gamma^-)$. Note that it has terminal position $X$ at time $t$, and hence
    \begin{equation*}
        \hat{U}(t,X) \leq \int_0^t \bar{L}\left( \frac{\gamma^+ + \gamma^-}{2}, \frac{\dot{\gamma}^+ + \dot{\gamma}^-}{2} \right) + \hat{\mathcal{F}}\left( \frac{\gamma^+ + \gamma^-}{2} \right) d\tau + \mathcal{U}_*\left( \frac{\gamma^+(0) + \gamma^-(0)}{2} \right)
    \end{equation*}
    
    It follows from the convexity of $\hat{B}_0^t$ that
    \begin{multline*}
        \hat{U}(t, X + h) + \hat{U}(t, X - h) \\
        = \int_0^t \bar{L}(\gamma^+, \dot{\gamma}^+) + \bar{L}(\gamma^-, \dot{\gamma}^-) + \hat{\mathcal{F}}(\gamma^+) + \hat{\mathcal{F}}(\gamma^-) \, d\tau + \mathcal{U}_*(\gamma^+(0)) + \mathcal{U}_*(\gamma^-(0)) \\
        \geq 2\int_0^t \bar{L}\left( \frac{\gamma^+ + \gamma^-}{2}, \frac{\dot{\gamma}^+ + \dot{\gamma}^-}{2} \right) + \hat{\mathcal{F}}\left( \frac{\gamma^+ + \gamma^-}{2} \right) d\tau + 2\,\mathcal{U}_*\left( \frac{\gamma^+(0) + \gamma^-(0)}{2} \right) \\
        \geq 2\hat{U}(t,X)
    \end{multline*}
    
    Thus we have that
    \begin{equation*}
        \hat{U}(t, X + h) + \hat{U}(t, X - h) - 2\hat{U}(t,X) \geq 0
    \end{equation*}
    
    Which implies
    \begin{equation}
        \label{lower_bound_uhat_2nd_derivative}
        \nabla_{XX} \hat{U}(t,X) \geq 0
    \end{equation}
    
    Combining (\ref{hard_case}) and (\ref{lower_bound_uhat_2nd_derivative}) results in
    \begin{equation}
        \label{hessian_uhat_bound}
        0 \leq \nabla_{XX} \hat{U}(t,X) \leq C_t
    \end{equation}
    
    By (\ref{hessian_uhat_bound}), we have that $\nabla_X \hat{U}$ exists and that $\nabla_X \hat{U}(t, \cdot)$ is $C_t$--Lipschitz (hence, continuous). We prove the following lemmas to aid us in guaranteeing continuity in the time input as well:
    
    \begin{lemma}
        \label{cont_of_seq_of_gamma}
        Let $\gamma_t^s[X](\omega)$ be the minimizer that obtains minimum $\hat{U}(t,X)$. Consider sequence of minimizers $(\gamma^s_t[X])_n(\omega)$ that obtain minimum $\hat{U}(t_n, X_n)$. Denoted $\gamma_n(s,\omega)$ $(\gamma_n : [0,T] \times \Omega \to \mathbb{H})$, it must satisfy the Euler-Lagrange Equation and following conditions:
        \begin{equation*}
            \begin{dcases}
                \frac{d}{ds} \nabla_v L(\gamma_n(s,\omega), \dot{\gamma}_n(s,\omega)) = \nabla_q L(\gamma_n(s,\omega), \dot{\gamma}_n(s,\omega)) + \int_{\Omega} \nabla \Phi (\gamma_n(s,\omega) - \gamma_n(s,\tilde{\omega})) \, d\tilde{\omega} \\
                \gamma_n(t,\omega) = X_n (\omega) \\
                \int_{\Omega} \nabla \Psi( \gamma_n(0,\omega) - \gamma_n(0,\tilde{\omega})) \, d\tilde{\omega} = \nabla_v L(\gamma_n(0,\omega), \dot{\gamma}_n(0,\omega))
            \end{dcases}
        \end{equation*}
        
        If $\lim_{n \to \infty} \gamma_n = \gamma_{\infty}$, then $\gamma_{\infty}$ is a minimizer and thus also satisfies the above (hence, the sequence $(\gamma_n)$ uniformly converges to $\gamma_{\infty}$).
    \end{lemma}
    
    \begin{proof}[Proof of Lemma \ref{cont_of_seq_of_gamma}.]
        First we utilize the following fact (let $E = [0,T] \times \Omega$):
        \begin{equation*}
            \int_{E} |\gamma_n(t,\omega)|^2 + |\dot{\gamma}_n(t,\omega)|^2 \, dt d\omega \leq \text{constant}
        \end{equation*}
        
        Therefore, up to a subsequence, $\gamma_n \rightharpoonup \gamma_{\infty}$, $\dot{\gamma}_n \rightharpoonup \dot{\gamma}_{\infty}$. Our goal is to show that given some minimizer $\gamma$ 
        \begin{equation*}
            \hat{U}(0, \gamma(0)) + \int_0^t \hat{L}(\gamma, \dot{\gamma}) \, ds \geq \hat{U}(0, \gamma_{\infty}(0)) + \int_0^t \hat{L}(\gamma_{\infty}, \dot{\gamma}_{\infty}) \, ds
        \end{equation*}
        
        We proceed by considering some path $\tilde{\gamma}_n$ knowing $\gamma(t) = X$, and set $\tilde{\gamma}_n(t) = X_n$ through: 
        \begin{equation*}
            \tilde{\gamma}_n(s) = 
            \begin{dcases}
                \gamma(s) & \text{if } 0 \leq s < t - \delta \\
                \gamma(s) + \frac{s - t + \delta}{\delta}(X_n - \gamma(s)) & \text{if } t - \delta \leq s \leq t
            \end{dcases}
        \end{equation*}
        
        As $\tilde{\gamma}_n$ is not completely optimal, we have
        \begin{align*}
            \hat{U}(0, \tilde{\gamma}_n(0)) + \int_0^t \hat{L}(\tilde{\gamma}_n, \dot{\tilde{\gamma}}_n) \, ds &\geq \hat{U}(0, \gamma_n(0)) + \int_0^{t_n} \hat{L}(\gamma_n, \dot{\gamma}_n) \, ds \\
            \hat{U}(0, \gamma(0)) + \int_0^{t - \delta} \hat{L}(\gamma, \dot{\gamma}) \, ds + \underset{\varepsilon}{\underbrace{\int_{t - \delta}^t \hat{L}(\tilde{\gamma}_n, \dot{\tilde{\gamma}}_n) \, ds}} &\geq \hat{U}(0, \gamma_n(0)) + \int_0^{t_n} \hat{L}(\gamma_n, \dot{\gamma}_n) \, ds \\
            \hat{U}(0, \gamma(0)) + \int_0^{t - \delta} \hat{L}(\gamma, \dot{\gamma}) \, ds + \varepsilon &\geq \hat{U}(0, \gamma_n(0)) + \int_0^{t_n} \hat{L}(\gamma_n, \dot{\gamma}_n) \, ds
        \end{align*}
        
        We utilize the convexity and thus weak lower semi-continuity of $\hat{\mathcal{B}}_0^t$ and take the limit as $n \to \infty$ on both sides to obtain
        \begin{align*}
            \hat{U}(0, \gamma(0)) + \int_0^{t - \delta} \hat{L}(\gamma, \dot{\gamma}) \, ds + \varepsilon &\geq \lim_{n \to \infty} \hat{U}(0, \gamma_n(0)) + \int_0^{t_n} \hat{L}(\gamma_n, \dot{\gamma}_n) \, ds \\
            &\geq \hat{U}(0, \gamma_{\infty}(0)) + \int_0^t \hat{L}(\gamma_{\infty}, \dot{\gamma}_{\infty}) \, ds
        \end{align*}
        
        By letting $\delta \to 0$, it follows that
        \begin{equation*}
            \hat{U}(0, \gamma(0)) + \int_0^t \hat{L}(\gamma, \dot{\gamma}) \, ds \geq \hat{U}(0, \gamma_{\infty}(0)) + \int_0^t \hat{L}(\gamma_{\infty}, \dot{\gamma}_{\infty}) \, ds
        \end{equation*}
        
        But as $\gamma$ is optimal, it follows that $\gamma_{\infty}$ is also optimal, and thus it is a minimizer.
    \end{proof}
    
    \begin{lemma}
        \label{derivative_hatL_derivative_L}
        For all $\omega \in \Omega$,
        \begin{equation*}
            \nabla_V \hat{L}(X(\omega), V(\omega)) = \nabla_v L(X(\omega), V(\omega))
        \end{equation*}
    \end{lemma}
        
    \begin{proof}[Proof of Lemma \ref{derivative_hatL_derivative_L}]
        Let $A,B \in \mathbb{H}$. Recall by definition we have
        \begin{equation}
            \label{hatL_exp_def}
            \hat{L}(X + A, V + B) = \int_{\Omega} L((X + A)(\omega), (V + B)(\omega)) \, d\omega + \frac{1}{2}\int_{\Omega^2} \Phi((X + A)(\omega) - (X + A)(\tilde{\omega})) \, d\omega d\tilde{\omega}
        \end{equation}
        
        Taylor's Theorem yields
        \begin{multline}
            \label{eqn_hatL_taylor}
            \hat{L}(X + A, V + B) = \hat{L}(X, V) + \langle \nabla_X \hat{L}(X(\omega),V(\omega)), A \rangle_{L^2(\Omega)} + \langle \nabla_V \hat{L}(X(\omega), V(\omega)), B \rangle_{L^2(\Omega)} \\
            + o(\| A \|) + o(\| B \|) 
        \end{multline}
        
        Where
        \begin{equation*}
            \langle P, Q \rangle_{L^2(\Omega)} = \int_{\Omega} \langle P(\omega), Q(\omega) \rangle \, d\omega
        \end{equation*}
        
        Likewise, 
        \begin{multline}
            \label{eqn_L_taylor}
            L((X + A)(\omega), (V + B)(\omega)) = L(X(\omega), V(\omega)) + \langle \nabla_q L(X(\omega), V(\omega)), A \rangle + \langle \nabla_v L(X(\omega), V(\omega)), B \rangle \\
            + o(\| A \|) + o(\| B \|) 
        \end{multline}
            
        Combining (\ref{hatL_exp_def}), (\ref{eqn_hatL_taylor}), and (\ref{eqn_L_taylor}),
        \begin{multline*}
            \hat{L}(X, V) + \langle \nabla_X \hat{L}(X(\omega), V(\omega)), A \rangle_{L^2(\Omega)} + \langle \nabla_V \hat{L}(X(\omega), V(\omega)), B \rangle_{L^2(\Omega)} + o(\| A \|) + o(\| B \|) = \\
            \left( \int_{\Omega} L(X(\omega), V(\omega)) \, d\omega + \frac{1}{2}\int_{\Omega^2} \Phi((X + A)(\omega) - (X + A)(\tilde{\omega})) \, d\omega d\tilde{\omega} \right) \\
            + \int_{\Omega} \langle \nabla_q L(X(\omega), V(\omega)), A \rangle \, d\omega + \int_{\Omega} \langle \nabla_v L(X(\omega), V(\omega)), B \rangle \, d\omega + o(\| A \|) + o(\| B \|) 
        \end{multline*}
        
        Which, for $A \equiv \mathbf{0}$, we obtain
        \begin{multline*}
            \hat{L}(X, V) + \langle \nabla_V \hat{L}(X(\omega), V(\omega)), B \rangle_{L^2(\Omega)} + o(\| B \|) = \left( \int_{\Omega} L(X(\omega), V(\omega)) \, d\omega \right. \\
            \left. + \frac{1}{2}\int_{\Omega^2} \Phi(X(\omega) - X(\tilde{\omega})) \, d\omega d\tilde{\omega} \right) + \int_{\Omega} \langle \nabla_v L(X(\omega), V(\omega)), B \rangle \, d\omega + o(\| B \|) 
        \end{multline*}
        
        Which we may rewrite as
        \begin{equation*}
            \hat{L}(X, V) + \langle \nabla_V \hat{L}(X(\omega), V(\omega)), B \rangle_{L^2(\Omega)} + o(\| B \|) = \hat{L}(X, V) + \langle \nabla_v L(X(\omega), V(\omega)), B \rangle_{L^2(\Omega)} + o(\| B \|)
        \end{equation*}
        
        By identification, we obtain
        \begin{equation*}
            \langle \nabla_V \hat{L}(X(\omega), V(\omega)), B \rangle_{L^2(\Omega)} = \langle \nabla_v L(X(\omega), V(\omega)), B \rangle_{L^2(\Omega)} 
        \end{equation*}
        
        As the above is true for all $B \in \mathbb{H}$, we obtain the desired result.
    \end{proof}
    
    We also prove the following claim:
    \begin{claim}
        \label{uhat_claim_X_derivative}
        We claim that
        \begin{equation*}
            \nabla_X \hat{U}(t, \gamma(t)) = \nabla_v L(\gamma(t), \dot{\gamma}(t))
        \end{equation*}
        
        Where $\gamma(s) = \gamma_t^s[X]$.
    \end{claim}
    \begin{proof}[Proof of Claim \ref{uhat_claim_X_derivative}]
        Let us introduce a variation on $\gamma$, $\tilde{\gamma}$ where
        \begin{equation*}
            \tilde{\gamma}(\tau) = \gamma(\tau) + \frac{\tau}{t}h \quad \left(\alpha(\tau) \overset{\text{set}}{=} \frac{\tau}{t} \right)
        \end{equation*}
        
        Note that $\tilde{\gamma}(0) = \gamma(0)$, $\tilde{\gamma}(t) = \gamma(t) + h$. Next, we define $A$:
        \begin{equation*}
            A(h) = \hat{U}(t, \gamma(t) + h) - \int_0^t \hat{L}(\gamma + \alpha h, \dot{\gamma} + \dot{\alpha}h) \, d\tau - \mathcal{U}_*(\gamma(0))
        \end{equation*}
        
        See that $A(h) \leq 0 \; \forall h$. Particularly, $A(\mathbf{0}) = 0 \implies h = \mathbf{0}$ is a maximizer of $A$. Therefore, $\nabla A(\mathbf{0}) = 0$ which implies
        \begin{equation*}
            0 = \nabla_X \hat{U}(t, \gamma(t)) \alpha(t) - \int_0^t \nabla_X \hat{L}(\gamma, \dot{\gamma})\alpha + \nabla_V \hat{L}(\gamma, \dot{\gamma})\dot{\alpha} \, d\tau
        \end{equation*}
        
        Via Integration by Parts, we rewrite
        \begin{align*}
            \int_0^t \nabla_V \hat{L}(\gamma, \dot{\gamma}) \dot{\alpha} \, d\tau &= \left.\nabla_V \hat{L}(\gamma(\tau), \dot{\gamma}(\tau)) \alpha(\tau)  \right]_{0}^{t} - \int_0^t \frac{d}{d\tau} \nabla_V \hat{L}(\gamma, \dot{\gamma}) \alpha \, d\tau \\
            &= \nabla_V \hat{L}(\gamma(t), \dot{\gamma}(t))\alpha(t) - \int_0^t \frac{d}{d\tau} \nabla_V \hat{L}(\gamma, \dot{\gamma}) \alpha \, d\tau 
        \end{align*}
        
        And hence
        \begin{equation*}
            0 = \nabla_X \hat{U}(t, \gamma(t))\alpha(t) - \nabla_V \hat{L}(\gamma(t), \dot{\gamma}(t))\alpha(t) - \int_0^t \left(\nabla_X \hat{L}(\gamma, \dot{\gamma}) - \frac{d}{d\tau} \nabla_V \hat{L}(\gamma, \dot{\gamma})\right)\alpha \, d\tau
        \end{equation*}
        
        By Proposition \ref{euler_lagrange_thm_uhat}, we get
        \begin{equation*}
            \int_0^t \left(\nabla_X \hat{L}(\gamma, \dot{\gamma}) - \frac{d}{d\tau} \nabla_V \hat{L}(\gamma, \dot{\gamma})\right)\alpha \, d\tau = 0
        \end{equation*}
        
        (You may confirm that the above expression is indeed the Euler--Lagrange equation.) Hence (after cancelling $\alpha(t)$),
        \begin{equation*}
            \nabla_X \hat{U} (t, \gamma(t)) = \nabla_V \hat{L}(\gamma(t), \dot{\gamma}(t))
        \end{equation*}
        
        Let $X = \gamma(t, \cdot)$ in Lemma \ref{derivative_hatL_derivative_L}. Then, the result follows.
    \end{proof}
    
    Thus, by Lemma \ref{cont_of_seq_of_gamma}, Claim \ref{uhat_claim_X_derivative} we have that as $\gamma_t^t[X] = X, \dot{\gamma}_t^t[X] = \dot{X}$, if $t_n \to t$
    \begin{align*}
        \lim_{n \to \infty} \nabla_X \hat{U}(t_n, X) &= \lim_{n \to \infty} \nabla_v L(\gamma_t^{t_n}[X], \dot{\gamma}^{t_n}_{t}[X]) \\
        &= \nabla_v L(X, \dot{X}) \\
        &= \nabla_v L(\gamma(t), \dot{\gamma}(t)) \\
        &= \nabla_X \hat{U}(t,X)
    \end{align*}
    
    This implies that $\nabla_X \hat{U}$ is also continuous in the time input, which concludes the proof.
\end{proof}

With this we may prove the existence and continuity of $\partial_t \hat{U}$:
\begin{theorem}
    \label{first_derivative_t_uhat}
    Given the existence, continuity of $\nabla_X \hat{U}(t,X)$, then $\partial_t\hat{U}(t,X)$ exists and is continuous.
\end{theorem}
\begin{proof}
    First we note that, for $\gamma(s,\omega) = \gamma_t^s[X](\omega)$, $\gamma(t) = X$ and thus
    \begin{equation*}
        \hat{U}(t,\gamma(t)) - \hat{U}(t - h, \gamma(t - h)) = \int_{t - h}^t \hat{L}(\gamma, \dot{\gamma}) \, d\tau
    \end{equation*}
    
    Our goal is to show that the following limit exists:
    \begin{equation*}
        \lim_{h \to 0} \frac{\hat{U}(t, \gamma(t)) - \hat{U}(t - h, \gamma(t))}{-h}
    \end{equation*}
    
    We note that by adding and subtracting $\hat{U}(t - h, \gamma(t))$
    \begin{equation*}
        [\hat{U}(t,\gamma(t)) - \hat{U}(t - h, \gamma(t))] + [\hat{U}(t - h, \gamma(t)) - \hat{U}(t - h, \gamma(t - h))] = \int_{t - h}^t \hat{L}(\gamma, \dot{\gamma}) \, d\tau
    \end{equation*}
    
    We divide by $-h$ on both sides, rearrange and obtain:
    \begin{equation}
    \label{uhat_difference_quotient}
        \frac{\hat{U}(t, \gamma(t)) - \hat{U}(t - h, \gamma(t))}{-h} = \frac{\hat{U}(t - h, \gamma(t)) - \hat{U}(t - h, \gamma(t - h))}{h} - \frac{1}{h}\int_{t - h}^t \hat{L}(\gamma, \dot{\gamma}) \, d\tau
    \end{equation}
    
    Using Taylor's Theorem,
    \begin{multline*}
        \frac{\hat{U}(t - h, \gamma(t)) - \hat{U}(t - h, \gamma(t - h))}{h} = \int_0^1 \biggl\langle \nabla_X \hat{U}(t - h, \gamma(t - h) + s(\gamma(t) - \gamma(t - h))),  \\
        \frac{\gamma(t) - \gamma(t - h)}{h} \biggr\rangle \, ds
    \end{multline*}
    
    In which, if $h \to 0$, by the continuity of $\nabla_X \hat{U}$ we have that
    \begin{align*}
        \lim_{h \to 0} \frac{\hat{U}(t - h, \gamma(t)) - \hat{U}(t - h, \gamma(t - h))}{h} &= \int_0^1 \langle \nabla_X \hat{U}(t, \gamma(t)), \dot{\gamma}(t) \rangle \, ds \\
        &= \langle \nabla_X \hat{U}(t, \gamma(t)), \dot{\gamma}(t) \rangle 
    \end{align*}
    
    And also note
    \begin{equation*}
        \lim_{h \to 0} -\frac{1}{h} \int_{t - h}^t \hat{L}(\gamma(\tau), \dot{\gamma}(\tau)) \, d\tau = -\hat{L}(\gamma(t), \dot{\gamma}(t))
    \end{equation*}
    
    Hence, letting $h \to 0$ in (\ref{uhat_difference_quotient}) gives us
    \begin{align*}
        \lim_{h \to 0} \frac{\hat{U}(t, \gamma(t)) - \hat{U}(t - h, \gamma(t))}{-h} = \langle \nabla_X \hat{U}(t, \gamma(t), \dot{\gamma}(t) \rangle - \hat{L}(\gamma(t), \dot{\gamma}(t))
    \end{align*}
    
    Hence, we conclude the limit exists. By definition, this limit simply differentiates with respect to time, hence we have demonstrated the existence of $\partial_t \hat{U}$. Furthermore, as we can write $\partial_t \hat{U}$ in terms of continuous functions, we conclude it is continuous, which concludes the proof.
\end{proof}

\subsection{The Hamilton--Jacobi Equation for \texorpdfstring{$\hat{U}$}{U-hat}}

Finally, with the regularity of $\hat{U}$, we may now demonstrate that it satisfies its HJE:
\begin{proposition}
    \label{uhat_hje}
    Given the existence of $\partial_t \hat{U}(t,X), \nabla_X \hat{U}(t,X) \in C$, we have
    \begin{equation*}
        \partial_t \hat{U}(t,X) + \hat{H}(X, \nabla_X \hat{U}(t,X)) = 0
    \end{equation*}
\end{proposition}
\begin{proof}
    First note that
    \begin{equation*}
        \nabla_v L(q, a) = b \iff \nabla_p H(q,b) = a
    \end{equation*}
    
    That is, $\nabla_v L(q, \cdot)$ and $\nabla_p H(q, \cdot)$ are inverses of each other. Set $\gamma(s, \cdot) \equiv \gamma_t^s[X](\cdot)$. Recall the Euler--Lagrange Equation
    \begin{equation}
        \label{euler_lagrange_eq_uhat}
        \frac{d}{ds} \nabla_v L (\gamma(s, \omega), \dot{\gamma}(s,\omega)) = \nabla_q L(\gamma(s,\omega), \dot{\gamma}(s,\omega)) + \int_{\Omega} \nabla \Phi(\gamma(s,\omega) - \gamma(s,\tilde{\omega})) \, d\tilde{\omega}
    \end{equation}
    
    We aim to rewrite this in a Hamiltonian System: 
    \begin{equation}
        \label{eqn_z_hsys_uhat}
        Z(s,\omega) \overset{\text{set}}{=} \nabla_v L(\gamma(s, \omega), \dot{\gamma}(s,\omega))
    \end{equation}
    
    Thus,
    \begin{equation}
        \label{eqn_dotgamma_hsys_uhat}
        \dot{\gamma}(s,\omega) = \nabla_p H(\gamma(s,\omega), Z(s,\omega))
    \end{equation}
    
    We now prove two useful lemmas.

    \begin{lemma}
        \label{uhat_hje_lemma_1}
        If $L(q_0, v_0) + H(q_0, p_0) = \langle v_0, p_0 \rangle$, then
        \begin{equation*}
            \nabla_q L(q_0, v_0) = -\nabla_q H(q_0, p_0)
        \end{equation*}
    \end{lemma}
    \begin{proof}[Proof of Lemma \ref{uhat_hje_lemma_1}.]
        Set
        \begin{equation*}
            f(q) = L(q, v_0) + H(q, p_0) - \langle v_0, p_0 \rangle
        \end{equation*}
        
        By the hypothesis, we have that
        \begin{equation*}
            f(q) \geq 0 \quad \forall q
        \end{equation*}
        
        As $f(q_0) = 0$, this implies that $q_0$ is a minimizer of $f$. That is,
        \begin{equation*}
            0 = \nabla f(q_0) = \nabla_q L(q_0, v_0) + \nabla_q H(q_0, p_0)
        \end{equation*}
        
        The result follows.
    \end{proof}
    
    \begin{lemma}
        \label{uhat_hje_lemma_2} $L(q_0, v_0) + H(q_0, p_0) = \langle v_0, p_0 \rangle$ if and only if
        \begin{equation*}
            p_0 = \nabla_v L(q_0,v_0) \text{ or } v_0 = \nabla_p H(q_0,p_0)
        \end{equation*}
    \end{lemma}
    \begin{proof}[Proof of Lemma \ref{uhat_hje_lemma_2}.]
        We prove both directions.
        
        \underline{($\implies$)}
        
        Recall
        \begin{align*}
            H(q_0, p_0) &= \sup_{v} \{ \langle v, p_0 \rangle - L(q_0, v) \} \\
            L(q_0, v_0) &= \sup_{p} \{ \langle v_0, p \rangle - H(q_0, p) \}
        \end{align*}
        
        Therefore, we define $f$ as
        \begin{equation*}
            f(v) \overset{\text{set}}{=} L(q_0, v) + H(q_0, p_0) - \langle v, p_0 \rangle \geq 0 \quad \forall v
        \end{equation*}
        
        By the hypothesis, at $v = v_0$, $f$ attains its minimum, i.e.,
        \begin{equation*}
            0 = f'(v_0) = \nabla_v L(q_0, v_0) - \langle 1, p_0 \rangle \implies p_0 = \nabla_v L(q_0, v_0)
        \end{equation*}
        
        Or instead, we may define $g$ as
        \begin{equation*}
            g(p) \overset{\text{set}}{=} L(q_0, v_0) + H(q_0, p) - \langle v_0, p \rangle \geq 0 \quad \forall p
        \end{equation*}
        
        By the hypothesis, at $p = p_0$, $g$ attains its minimum, i.e.,
        \begin{equation*}
            0 = g'(p_0) = \nabla_p H(q_0, p_0) - \langle v_0, 1 \rangle \implies v_0 = \nabla_p H(q_0, p_0)
        \end{equation*}
        
        Which proves the forward direction.
        
        \newpage
        \underline{($\impliedby$)}
        
        Suppose $v_0 = \nabla_p H(q_0, p_0)$. Define $h$ as
        \begin{equation*}
            h(v, p) = L(q_0, v) + H(q_0, p) - \langle v, p \rangle
        \end{equation*}
        
        Then
        \begin{align*}
            \nabla_p h(v, p) &= \nabla_p H(q_0, p) - \langle v, 1 \rangle \\ 
            &= \nabla_p H(q_0, p) - v \\
            \nabla_p h(v_0, p) &= \nabla_p H(q_0, p) - \nabla_p H(q_0, p_0)
        \end{align*}
        
        Hence $p_0$ is a critical point of $p \mapsto h(v_0, p)$. Therefore, for all $p$ we can conclude
        \begin{align*}
            h(v_0, p_0) &\leq h(v_0, p) \\
            L(q_0, v_0) + H(q_0, p_0) - \langle v_0, p_0 \rangle &\leq L(q_0, v_0) + H(q_0, p) - \langle v_0, p \rangle \\
            \langle v_0, p_0 \rangle - H(q_0, p_0) &\geq \langle v_0, p \rangle - H(q_0, p)
        \end{align*}
        
        Which implies that
        \begin{align*}
            \langle v_0, p_0 \rangle - H(q_0, p_0) &= \sup_{p} \{ \langle v_0, p \rangle - H(q_0, p) \} \\
            &= L(q_0, v_0)
        \end{align*}
        
        Or, if we supposed that $p_0 = \nabla_v L(q_0, v_0)$, then
        \begin{align*}
            \nabla_v h(v, p) &= \nabla_v L(q_0, v) - \langle 1, p \rangle \\
            &= \nabla_v L(q_0, v) - p \\
            \nabla_v h(v, p_0) &= \nabla_v L(q_0, v) - \nabla_v L(q_0, v_0)
        \end{align*}
        
        Hence $v_0$ is a critical point of $v \mapsto h(v, p_0)$. Therefore for all $v$ we can conclude
        \begin{align*}
            h(v_0, p_0) &\leq h(v, p_0) \\
            L(q_0, v_0) + H(q_0, p_0) - \langle v_0, p_0 \rangle &\leq L(q_0, v) + H(q_0, p_0) - \langle v, p_0 \rangle \\
            \langle v_0, p_0 \rangle - L(q_0, v_0) &\geq \langle v, p_0 \rangle - L(q_0, v) 
        \end{align*}
        
        Which implies that
        \begin{align*}
            \langle v_0, p_0 \rangle - L(q_0, p_0) &= \sup_{v} \{ \langle v, p_0 \rangle - L(q_0, v) \} \\
            &= H(q_0, v_0)
        \end{align*}
        
        Which proves the backward direction.
    \end{proof}
    
    By Lemma \ref{uhat_hje_lemma_2}, (\ref{eqn_z_hsys_uhat}) implies
    \begin{equation*}
        L(\gamma, \dot{\gamma}) + H(\gamma, Z) = \langle \dot{\gamma}, Z \rangle
    \end{equation*}
    
    Which, when combined with Lemma \ref{uhat_hje_lemma_1},
    \begin{equation}
        \label{lagrange_to_hamiltonian_uhat}
        \nabla_q L (\gamma, \dot{\gamma}) = -\nabla_q H(\gamma, Z)
    \end{equation}
    
    Thus, by (\ref{euler_lagrange_eq_uhat}), (\ref{eqn_z_hsys_uhat}), (\ref{eqn_dotgamma_hsys_uhat}), and (\ref{lagrange_to_hamiltonian_uhat}) we have the Hamiltonian system
    \begin{equation*}
        \begin{dcases}
            \dot{Z}(s, \omega) = -\nabla_q H(\gamma_t^s[X](\omega), Z(s, \omega)) + \int_{\Omega} \nabla \Phi(\gamma(s, \omega) - \gamma(s, \tilde{\omega})) \, d\tilde{\omega} \\
            \dot{\gamma}_t^s[X](\omega) = \nabla_p H(\gamma_t^s[X](\omega), Z(s,\omega))
        \end{dcases}
    \end{equation*}
    
    Again, we denote
    \begin{equation*}
        \hat{\mathcal{F}}(X) = \frac{1}{2} \int_{\Omega^2} \Phi(X(\omega) - X(\tilde{\omega})) \, d\omega d\tilde{\omega}
    \end{equation*}
    
    And we note that
    \begin{equation*}
        \hat{U}(t, \bar{\gamma}(t)) \leq \int_0^t\int_{\Omega} L(\bar{\gamma}, \dot{\bar{\gamma}}) \, d\omega + \hat{\mathcal{F}}(\gamma) \, d\tau + \hat{U}(0, \bar{\gamma}(0)) \quad \forall \bar{\gamma}, t \in [0,T]
    \end{equation*}
    
    Rearranging,
    \begin{equation}
        \label{uhat_difference_estimation}
        \hat{U}(t, \bar{\gamma}(t)) - \int_0^t\int_{\Omega} L(\bar{\gamma}, \dot{\bar{\gamma}}) \, d\omega + \hat{\mathcal{F}}(\gamma) \, d\tau - \hat{U}(0, \bar{\gamma}(0)) \leq 0
    \end{equation}
    
    And see that if we plug in $\bar{\gamma} = \gamma$, we can define $f$ such that
    \begin{equation*}
        f(s) \overset{\text{set}}{=} \hat{U}(s, \gamma(s)) - \int_0^s\int_{\Omega} L(\gamma, \dot{\gamma}) \, d\omega + \hat{\mathcal{F}}(\gamma) \, d\tau - \hat{U}(0, \gamma(0)) = 0
    \end{equation*}
    
    Hence, $f'(s) = 0$ and thus
    \begin{equation}
        \label{derivative_f}
        f'(s) = \partial_s \hat{U}(s, \gamma(s)) + \langle \nabla_X \hat{U}(s,\gamma(s)), \dot{\gamma}(s) \rangle - \int_{\Omega} L(\gamma, \dot{\gamma}) \, d\omega - \hat{\mathcal{F}}(\gamma) = 0
    \end{equation}
    
    Combining (\ref{derivative_f}), Lemma \ref{uhat_hje_lemma_2}, and Claim \ref{uhat_claim_X_derivative} we conclude
    \begin{equation*}
        \partial_s \hat{U}(s, \gamma(s)) + \int_{\Omega} H(\gamma(s), \nabla_X \hat{U} (s, \gamma(s))) \, d\omega - \hat{\mathcal{F}}(\gamma(s)) = 0 \quad \forall s
    \end{equation*}
    
    Take $s = t$, and note that $\gamma(t) = \gamma_t^t[X] = X$ (by the conditions on the Euler--Lagrange Equation from earlier). This yields
    \begin{equation*}
        \partial_t \hat{U}(t, X) + \int_{\Omega} H(X, \nabla_X \hat{U}(t, X)) \, d\omega - \hat{\mathcal{F}}(X) = 0
    \end{equation*}
    
    Or in other words
    \begin{equation*}
        \partial_t \hat{U}(t, X) + \hat{H}(X, \nabla_X \hat{U}(t, X)) = 0
    \end{equation*}
    
    Which concludes the proof.
\end{proof}
\section{The Existence and Uniqueness of Optimal Individual Cost}

\subsection{Existence and Uniqueness}
Next we shift our discussion to an individual player's cost. Consider the function
\begin{equation*}
    B_0^T(r) = \int_0^T L(r(s), \dot{r}(s)) + F(r(s), \gamma_T^s[X]) \, ds + G(r(0), \gamma_T^0[X])
\end{equation*}

Where we define, $\forall q \in \mathbb{R}^d, \forall X \in \mathbb{H}$,
\begin{align*}
    F(q, X) \coloneqq \int_{\Omega} \Phi(q - X(\tilde{\omega})) \, d\tilde{\omega} \\
    G(q, X) \coloneqq \int_{\Omega} \Psi(q - X(\tilde{\omega})) \, d\tilde{\omega} 
\end{align*}

See that $B_0^t$ effectively calculates an individual player's cost along some path $r$. We seek to minimize such a cost: our goal is to again prove the existence of a minimum for $B_0^T$.

\begin{theorem}
    \label{exist_littleuhat_thm}
    Given $B_0^T$ as defined above, we have that a minimizer exists: we have a minimum $\hat{u}$ where
    \begin{equation*}
        \hat{u}(T,q;X) = \inf_{r} \left\{ \int_0^T L(r(s), \dot{r}(s)) + F(r(s), \gamma_T^s[X]) \, ds + G(r(0), \gamma_T^0[X]) : r(T) = q \right\}
    \end{equation*}
\end{theorem}
\begin{proof}
    Again, we proceed via the Direct Method in the Calculus of Variations. Let $(r_n)$ be a minimizing sequence on $B_0^T$:
    \begin{equation*}
        B_0^T(r_n) \to \inf_{\bar{r}}\left\{ B_0^T(\bar{r}) \right\} \overset{\text{set}}{=} w
    \end{equation*}
    
    By definition, we have
    \begin{multline*}
        w + 1 \geq B_0^T (r_n) = \int_0^T L(r_n(s), \dot{r}_n(s)) \, ds + \int_0^T \int_{\Omega} \Phi(r_n(s) - \gamma_t^s[X](\omega)) \, d\tilde{\omega} ds \\
        + \int_{\Omega} \Phi(r_n(0) - \gamma_t^0[X](\omega)) \, d\tilde{\omega} \\
        \geq c_0\int_0^T |\dot{r}_n(s)|^2 \, ds - D_1 - D_2
    \end{multline*}
    
    Via our initial assumptions (i.e., bounds on $L, \Phi, \Psi$). Again, via the Poincaré Inequality, there exists some constant $D > 0$ such that
    \begin{equation*}
        c_0 \int_0^T |\dot{r}_n(s)|^2 \, ds \geq D \int_0^T |r_n(s)|^2 \, ds
    \end{equation*}
    
    Giving us
    \begin{align*}
        D\int_0^T |r_n(s)|^2 \, ds - D_1 - D_2 &\leq w + 1 \\
        \left\| r_n \right\|_{L^2}^2 = \int_0^T |r_n(s)|^2 \, ds  &\leq \frac{1}{D}(w + D_1 + D_2 + 1) \overset{\text{set}}{=} M^2
    \end{align*}
    
    Note that $M^2 > 0$ does not depend on $n$. Square-rooting both sides yields:
    \begin{equation*}
    \left\| r_n \right\|_{L^2} \leq M \quad \forall n
    \end{equation*}
    
    Thus, we obtain the boundedness of $(r_n)$. Recall that in $L^2$, every bounded sequence is weakly pre-compact, and conclude that our minimizing sequence is weakly pre-compact: that is, it admits a weakly convergent subsequence $(r_{n_k})$ to some $r_0$:
    
    \begin{equation*}
        \lim_{k \to \infty} \langle r_{n_k}, g \rangle = \langle r_0 , g \rangle \quad \forall g \in L^2
    \end{equation*}
    
    Next we aim to demonstrate the convexity of $B_0^T$. Consider some $r,\Bar{r} \in \text{Dom}(B_0^T)$ where $r(T) = \bar{r}(T)$. Consider interpolation $r^{\varepsilon} = (1 - \varepsilon)r + \varepsilon\bar{r}$ and set $g(\varepsilon) = B_0^T(r^{\varepsilon})$:
    \begin{multline*}
        g(\varepsilon) = \int_0^T L(r^{\varepsilon}(s), \dot{r}^{\varepsilon}(s)) \, ds + \int_0^T \int_{\Omega} \Phi(r^{\varepsilon}(s) - \gamma_t^s[X](\omega)) \, d\tilde{\omega} ds \\
        + \int_{\Omega} \Psi(r^{\varepsilon}(0) - \gamma_t^0[X](\omega)) \, d\tilde{\omega} 
    \end{multline*}
    
    If we differentiate twice,
    \begin{multline*}
        g''(\varepsilon) = \int_0^T \left\langle \nabla^2 L(r^{\varepsilon}, \dot{r}^{\varepsilon}) \begin{pmatrix} \bar{r} - r \\ \dot{\bar{r}} - \dot{r} \end{pmatrix}, \begin{pmatrix} \bar{r} - r \\ \dot{\bar{r}} - \dot{r} \end{pmatrix} \right\rangle \, ds \\
        + \int_0^T\int_{\Omega} \left\langle \nabla^2 \Phi(r^{\varepsilon}(s) - \gamma_T^s[X](\tilde{\omega}))(\bar{r} - r), \bar{r} - r \right\rangle \, d\tilde{\omega} ds\\
        + \int_{\Omega} \left\langle \nabla^2 \Psi(r^{\varepsilon}(0) - \gamma_T^0[X](\tilde{\omega}))(\bar{r}(0) - r(0)), \bar{r}(0) - r(0) \right\rangle \, d\tilde{\omega}
    \end{multline*}
    
    And by our initial assumptions we have
    \begin{align*}
        g''(\varepsilon) &\geq \int_0^T \lambda_0 |\dot{\bar{r}} - \dot{r}|^2 \, ds + \int_0^T\int_{\Omega} c_{\Phi}|\bar{r} - r|^2 \, d\tilde{\omega} ds + \int_{\Omega} c_{\Psi} |\bar{r}(0) - r(0)|^2 \, d\tilde{\omega} \\
        &= \int_0^T \lambda_0 |\dot{\bar{r}} - \dot{r}|^2 \, ds + \int_0^T c_{\Phi}|\bar{r} - r|^2 \, ds + c_{\Psi}|\bar{r}(0) - r(0)|^2 \\
        &= \int_0^T \lambda_0 |\dot{\bar{r}} - \dot{r}|^2 \, ds + \int_0^T (c^{+}_{\Phi} - c^{-}_{\Phi})|\bar{r} - r|^2 \, ds + (c^{+}_{\Psi} - c^{-}_{\Psi})|\bar{r}(0) - r(0)|^2
    \end{align*}
    
    By the Poincaré Inequality,
    \begin{align*}
        \frac{\lambda_0}{2} \int_0^T |\dot{\bar{r}} - \dot{r}|^2 \, ds &\geq \frac{\lambda_0}{2} \cdot \frac{2}{T^2} \int_0^T |\bar{r} - r|^2 \, ds \\
        &= \frac{\lambda_0}{T^2} \int_0^T |\bar{r} - r|^2 \, ds
    \end{align*}
    
    And also,
    \begin{align*}
        \frac{\lambda_0}{2} \int_0^T |\dot{\bar{r}} - \dot{r}|^2 \, ds &\geq \frac{\lambda_0}{2} \cdot \frac{1}{T} |\bar{r}(0) - r(0)|^2 \\
        &= \frac{\lambda_0}{2T} |\bar{r}(0) - r(0)|^2
    \end{align*}
    
    Yielding
    \begin{align*}
        g''(\varepsilon) &\geq \int_0^T \lambda_0 |\dot{\bar{r}} - \dot{r}|^2 \, ds + \int_0^T (c^{+}_{\Phi} - c^{-}_{\Phi})|\bar{r} - r|^2 \, ds + (c^{+}_{\Psi} - c^{-}_{\Psi})|\bar{r}(0) - r(0)|^2 \\
        &= 2\left( \frac{\lambda_0}{2}\int_0^T |\dot{\bar{r}} - \dot{r}|^2 \, ds \right) + \int_0^T (c^{+}_{\Phi} - c^{-}_{\Phi})|\bar{r} - r|^2 \, ds + (c^{+}_{\Psi} - c^{-}_{\Psi})|\bar{r}(0) - r(0)|^2 \\
        &\geq \int_0^T \left(\frac{\lambda_0}{T^2} + c^{+}_{\Phi} - c^{-}_{\Phi}\right)|\bar{r} - r|^2 \, ds + \left(\frac{\lambda_0}{2T} + c^{+}_{\Psi} - c^{-}_{\Psi}\right)|\bar{r}(0) - r(0)|^2
    \end{align*}
    
    Which is certainly nonnegative for small enough $T$. Hence $g'' \geq 0 \iff g \text{ convex} \implies B_0^T \text{ convex} \implies B_0^T \text{ lower semi-continuous}$. And thus
    \begin{equation*}
        \inf_{\bar{r}} \left\{ B_0^{T}(\bar{r}) \right\} = \lim_{n \to \infty} B_0^T (r_n) = \lim_{n \to \infty} B_0^T (r_{n_k}) \geq B_0^T (r_0) \geq \inf_{\bar{r}} \left\{ B_0^{T}(\bar{r}) \right\}
    \end{equation*}
    
    From which we conclude that a minimum exists for some minimizer $r_0$:
    \begin{equation*}
        \hat{u}(T,q;X) = \inf_{\bar{r}} \left\{ B_0^{T}(\bar{r}) : \bar{r}(T) = q \right\}
    \end{equation*}
    
    We now show that such a minimizer is unique. Assume we have two distinct minimizers of $B_0^T$ $r^0 \neq r^1$. Like before, we aim to show that they are equal. We consider a convex combination of these minimizers $r^\varepsilon = (1 - \varepsilon)r^0 - \varepsilon r^1$. Taylor's Theorem on $B_0^T$ yields
    \begin{equation*}
        B_0^T(r^1) - B_0^T(r^0) = \left. \frac{d}{d\varepsilon} B_0^T(r^\varepsilon) \right|_{\varepsilon = 0} + \frac{1}{2} \left. \frac{d^2}{d\varepsilon^2} B_0^T(r^{\varepsilon}) \right|_{\varepsilon = \varepsilon^*}
    \end{equation*}
    
    $r^1$ and $r^0$ both being minimizers implies that $B_0^T (r^1) - B_0^T(r^0) = 0$ and $\left. \frac{d}{d\varepsilon} B_0^T(r^\varepsilon) \right|_{\varepsilon = 0} = 0$. Thus we conclude that $\left. \frac{d^2}{d\varepsilon^2} B_0^T(r^{\varepsilon}) \right|_{\varepsilon = \varepsilon^*} = 0$. From above we have that
    \begin{equation*}
        g''(\varepsilon) \geq \left(\lambda_0 - \frac{T^2}{2}c_{\Phi}^{-} - Tc_{\Psi}^{-} \right) \int_0^T |\dot{r}^1 - \dot{r}^0|^2 \, dt \geq 0
    \end{equation*}
    
    As the constant multiple is again strictly positive, we conclude that if $g''(\varepsilon^*) = 0$ then
    \begin{equation*}
        \int_0^T |\dot{r}^1 - \dot{r}^0|^2 \, dt = 0 \implies \dot{r}^1 = \dot{r}^0
    \end{equation*}
    
    As the terminal position for all minimizers is $q$ ($r(T) = q$), we conclude that as $r^1$ and $r^0$ have equal derivatives and agree at a point, they must be equal everywhere: $r^1 = r^0$, and hence the minimizer is unique. Furthermore, it follows that if $q = X(\omega)$, as $s \mapsto \gamma^s_t[X](\omega)$ is a solution to the Euler--Lagrange Equation for the minimizer of $\hat{u}$, that it is also the unique minimizer of $\hat{u}$. 
\end{proof}

\subsection{Euler--Lagrange Equation of Minimizer}

\begin{proposition}
    \label{euler_lagrange_thm_u}
    $r(s) = {\gamma}_t^s[X](\omega)$ is a minimizer achieving $\hat{u}(t,q,X)$ if and only if 
    \begin{equation*}
        \begin{dcases}
            \nabla_q L(r(s), \dot{r}(s)) - \frac{d}{ds} \nabla_v L(r(s), \dot{r}(s)) + \int_{\Omega} \nabla\Phi (r(s) - \gamma_t^s[X](\tilde{\omega})) \, d\tilde{\omega} = 0 \\
            r(t) = X(\omega) \\
            \int_{\Omega} \nabla \Psi(r(0) - \gamma_t^0[X](\tilde{\omega})) \, d\tilde{\omega} = \nabla_v L(r(0), \dot{r}(0))
        \end{dcases}
    \end{equation*}
\end{proposition}
\begin{proof}
    Let $\omega$ be such that $q = X(\omega) \; (= \gamma_t^t [X](\omega))$, in which $\gamma_t^s[X](\omega)$ is the unique minimizer achieving $\hat{u}$. Consider 
    \begin{equation*}
        I(\bar{r}) = \int_0^t L(\bar{r}, \dot{\bar{r}}) + F(\bar{r}, \gamma_t^s[X]) \, ds + G(\bar{r}(0), \gamma_t^0 [X])
    \end{equation*}

    Suppose $\bar{r}(t) = q = X(\omega)$. Set $r(s) = \gamma_t^s[X](\omega)$ to be the minimizer of $I$ ($I(r) = \hat{u}(t,q,X)$). Consider a variation on $r$: let $y$ be such that $y(0) = y(t) = 0$. Let $\varepsilon \in \mathbb{R}$. Then
    \begin{align*}
        (r + \varepsilon y)(0) &= r(0) \\
        (r + \varepsilon y)(t) &= r(t)
    \end{align*}

    Let $b(\varepsilon) = I(r + \varepsilon y)$. Then
    \begin{align*}
        b(\varepsilon) &= \int_0^t L(r + \varepsilon y, \dot{r} + \varepsilon \dot{y}) + F(r + \varepsilon y, \gamma_t^s[X]) \, ds + G(\bar{r}(0), \gamma_t^0 [X]) \\
        b'(\varepsilon) &= \int_0^t \langle \nabla_q L(r + \varepsilon y, \dot{r} + \varepsilon \dot{y}), y \rangle + \langle \nabla_v L(r + \varepsilon y, \dot{r} + \varepsilon \dot{y}), \dot{y} \rangle + \langle \nabla F(r + \varepsilon y, \gamma_t^s [X]), y \rangle \, ds 
    \end{align*}

    As $r$ minimizes $I$, we note that $b$ attains its minimum at $\varepsilon = 0$, and thus $b'(0) = 0$. In other words
    \begin{equation*}
        0 = b'(0) = \int_0^t \langle \nabla_q L(r, \dot{r}), y \rangle + \langle \nabla_v L(r, \dot{r}), \dot{y} \rangle + \langle \nabla F(r, \gamma_t^s [X]), y \rangle \, ds
    \end{equation*}

    Via Integration by Parts,
    \begin{equation*}
        \int_0^t \langle \nabla_v L(r, \dot{r}), \dot{y} \rangle \, ds = \biggl. \langle \nabla_v L(r, \dot{r}), y \rangle \biggr]_0^t - \int_0^t \left\langle \frac{d}{ds} \nabla_v L(r, \dot{r}), y \right\rangle \, ds
    \end{equation*}

    In which $\bigl. \langle \nabla_v L(r, \dot{r}), y \rangle \bigr]_0^t = \langle \nabla_v L (r(t), \dot{r}(t) , 0 \rangle - \langle \nabla_v L (r(0), \dot{r}(0) , 0 \rangle = 0$. Therefore
    \begin{equation*}
        \int_0^t \langle \nabla_v L(r, \dot{r}), \dot{y} \rangle \, ds = \int_0^t \left\langle -\frac{d}{ds} \nabla_v L(r, \dot{r}), y \right\rangle \, ds
    \end{equation*}

    Hence,
    \begin{align*}
        0 &= \int_0^t \langle \nabla_q L(r, \dot{r}), y \rangle + \left\langle -\frac{d}{ds} \nabla_v L(r, \dot{r}), y \right\rangle + \langle \nabla F(r, \gamma_t^s [X]), y \rangle \, ds \\
        &= \int_0^t \left\langle \nabla_q L(r, \dot{r}) - \frac{d}{ds} \nabla_v L(r, \dot{r}) + \nabla F(r, \gamma_t^s [X]), y \right\rangle \, ds
    \end{align*}

    As the above is true for all $y$ on $[0,t]$, we conclude that the first argument in the inner product must be 0. We conclude
    \begin{equation*}
        \frac{d}{ds} \nabla_v L(r, \dot{r}) = \nabla_q L(r, \dot{r}) + \nabla F(r, \gamma_t^s[X]) \int_{\Omega} \nabla \Phi(r(s) - \gamma_t^s[X](\tilde{\omega})) \, d\tilde{\omega}
    \end{equation*}

    Where again, $r(s) = \gamma_t^s[X](\omega)$. We show that
    \begin{equation*}
        \int_{\Omega} \nabla \Psi(r(0) - \gamma_t^0[X](\tilde{\omega})) \, d\tilde{\omega} = \nabla_v L(r(0), \dot{r}(0))
    \end{equation*}

    By plugging in $t = 0$ in Claim \ref{u_hje_claim_1} and thus obtain the desired result.
\end{proof}
\section{Differentiability of \texorpdfstring{$u$}{u} and its Hamilton--Jacobi Equation}

First, define $u$ as
\begin{equation*}
    u(t,q) = \hat{u}(t,q, \gamma_T^t[X]) \quad \forall t \in [0,T]
\end{equation*}

Hence, as $u$ is defined by $\hat{u}$, everything we demonstrate for $\hat{u}$ also follows for $u$. We aim to show that $u$ is continuously differentiable (existence of $\partial_t u, \nabla_q u \in C$) and that it satisfies the Hamilton--Jacobi Equation
\begin{equation*}
    \partial_t u(t,q) + H(q, \nabla_q u(t,q)) - \int_{\Omega} \Phi(q - \gamma_T^t[X](\tilde{\omega})) \, d\tilde{\omega} = 0
\end{equation*}

With the boundary condition
\begin{equation*}
    u(0,q) = G(q, \gamma_T^0[X])
\end{equation*}

\subsection{Existence \& Continuity of \texorpdfstring{$\partial_t u$}{Partial-t u}, \texorpdfstring{$\nabla_q u$}{Nabla-q u}}

Just like before, we use a second difference quotient argument to establish the existence and continuity of the space derivative of $u$ ($\nabla_q u$), and then use this to imply the existence and continuity of the time derivative of $u$ ($\partial_t u$).

\begin{theorem}
    \label{first_derivative_q_u}
    $\nabla_q u(t,q)$ exists and is continuous.
\end{theorem}

\begin{proof}
    Recall
    \begin{equation*}
        \hat{u}(t,q,X) = \inf_{\bar{r}} \left\{ \int_0^t L(\bar{r}, \dot{\bar{r}}) + F(\bar{r}, \gamma_t^s[X]) \, ds + G(\bar{r}(0), \gamma_t^0[X]) : \bar{r}(t) = q \right\}
    \end{equation*}
    
    Let $r$ be optimal in $\hat{u}$ ($\bar{r} = r(s) = \gamma_t^s[X](\omega)$ for $q = X(\omega)$). That is,
    \begin{equation*}
        \hat{u}(t,q,X) = \int_0^t L(r, \dot{r}) + F(r, \gamma_t^s[X]) \, ds + G(r(0), \gamma_t^0[X]) \quad (r(t) = q)
    \end{equation*} 
    
    We introduce a variation on $r$: let $\tilde{r}(\tau) = r(\tau) + \frac{\tau}{t}h$, hence $\tilde{r}(0) = r(0)$ and $\tilde{r}(t) = r(t) + h$. Therefore,
    \begin{align*}
        \hat{u}(t, q + h, X) &\leq \int_0^t L\left(r + \frac{\tau}{t}h, \dot{r} + \frac{h}{t} \right) + F\left(r + \frac{\tau}{t}h, \gamma_t^\tau[X]\right) \, d\tau + G(r(0), \gamma_t^0[X]) \\
        \hat{u}(t, q - h, X) &\leq \int_0^t L\left(r - \frac{\tau}{t}h, \dot{r} - \frac{h}{t} \right) + F\left(r - \frac{\tau}{t}h, \gamma_t^\tau[X]\right) \, d\tau + G(r(0), \gamma_t^0[X])
    \end{align*}
    
    Let $C_L$, $C_F$ be such that $\nabla^2 L \leq C_L I$, $\nabla^2 F \leq C_F I$. Like before, we have
    \begin{align*}
        \hat{u}(t, q + h, X) + \hat{u}(t, q - h, X) - 2\hat{u}(t, q, X) &\leq \int_0^t C_L \left\langle \begin{pmatrix} \frac{\tau h}{t} \\ \frac{h}{t} \end{pmatrix}, \begin{pmatrix} \frac{\tau h}{t} \\ \frac{h}{t} \end{pmatrix} \right\rangle + C_F \left\langle \frac{\tau h}{t}, \frac{\tau h}{t} \right\rangle \, d\tau \\
        &= \int_0^t C_L \left(\frac{\tau^2 |h|^2}{t^2} + \frac{|h|^2}{t^2} \right) + C_F \left( \frac{\tau^2 |h|^2}{t^2} \right) \, d\tau \\
        &= \left. (C_L + C_F)\frac{\tau^3}{3t^2}|h|^2 + C_L \frac{\tau}{t^2}|h|^2 \right]_{\tau = 0}^{\tau = t} \\
        &= \left(\frac{t}{3}(C_L + C_F) + \frac{1}{t}C_L\right)|h|^2
    \end{align*}
    
    Which implies that
    \begin{equation}
        \label{hard_case_u}
        \nabla_{qq} \hat{u}(t, q, X) \leq \left(\frac{t}{3}(C_L + C_F) + \frac{1}{t}C_L\right) \overset{\text{set}}{=} c_t
    \end{equation}
    
    For a lower bound on $\nabla_{qq} \hat{u}$, we simply consider optimal paths on its variations: let $r^+$, $r^-$ denote the optimal paths for $\hat{u}(t, q + h, X)$ and $\hat{u}(t, q - h, X)$, respectively. Thus,
    \begin{align*}
        \hat{u}(t, q + h, X) &= \int_0^t L(r^+, \dot{r}^+) + F(r^+, \gamma_t^\tau[X]) \, d\tau + G(r^+(0), \gamma_t^0[X]) \\
        \hat{u}(t, q - h, X) &= \int_0^t L(r^-, \dot{r}^-) + F(r^-, \gamma_t^\tau[X]) \, d\tau + G(r^-(0), \gamma_t^0[X])
    \end{align*}
    
    Consider the path obtained by taking the average of these two optimal paths, $\frac{1}{2}(r^+ + r^-)$. Note that it has terminal position $q$ at time $t$, and hence
    \begin{equation*}
        \hat{u}(t, q, X) \leq \int_0^t L \left( \frac{r^+ + r^-}{2}, \frac{\dot{r}^+ + \dot{r}^-}{2} \right) + F\left( \frac{r^+ + r^-}{2}, \gamma_t^\tau[X] \right) \, d\tau + G\left(\frac{r^+(0) + r^-(0)}{2}, \gamma_t^0[X] \right)
    \end{equation*}
    
    It follows from the convexity of $B_0^t$ that
    \begin{multline*}
        \hat{u}(t, q + h, X) + \hat{u}(t, q - h, X) \\
        = \int_0^t L(r^+, \dot{r}^+) + L(r^-, \dot{r}^-) + F(r^+, \gamma_t^\tau[X]) + F(r^-, \gamma_t^\tau[X]) \, d\tau + G(r^+(0), \gamma_t^0[X]) + G(r^-(0), \gamma_t^0[X]) \\
        \geq 2\int_0^t L \left( \frac{r^+ + r^-}{2}, \frac{\dot{r}^+ + \dot{r}^-}{2} \right) + F\left( \frac{r^+ + r^-}{2}, \gamma_t^\tau[X] \right) \, d\tau + 2G\left(\frac{r^+(0) + r^-(0)}{2}, \gamma_t^0[X] \right) \\
        \geq 2\hat{u}(t, q, X)
    \end{multline*}
    
    Therefore we obtain
    \begin{equation*}
        \hat{u}(t, q + h, X) + \hat{u}(t, q - h, X) - 2\hat{u}(t, q, X) \geq 0
    \end{equation*}
    
    Hence
    \begin{equation}
        \label{lower_bound_u_2nd_derivative}
        \nabla_{qq} \hat{u}(t,q,X) \geq 0
    \end{equation}
    
    Combining (\ref{hard_case_u}) and (\ref{lower_bound_u_2nd_derivative}) results in
    \begin{equation}
        \label{hessian_u_bound}
        0 \leq \nabla_{qq} \hat{u}(t, q, X) \leq c_t
    \end{equation}
    
    By (\ref{hessian_u_bound}), we have that $\nabla_{qq} \hat{u}$ exists and that $\nabla_{qq} \hat{u}(t, \cdot, X)$ is $c_t$--Lipschitz. From this we may conclude that $\nabla_{qq} u$ exists and that $\nabla_{qq} u(t, \cdot)$ is $c_t$--Lipschitz (hence, continuous). We prove the following claim:
    
    \begin{claim}
        \label{u_hje_claim_1}
        We claim that
        \begin{equation*}
            \nabla_q u(t, r(t)) = \nabla_v L(r(t), \dot{r}(t))
        \end{equation*}
    \end{claim}
    \begin{proof}[Proof of Claim \ref{u_hje_claim_1}.]
        Let us introduce variation on $r$, $\tilde{r}$ where
        \begin{equation*}
            \tilde{r}(\tau) = r(\tau) + \frac{\tau}{t}h \quad \left(\alpha(\tau) \overset{\text{set}}{=} \frac{\tau}{t}\right)
        \end{equation*}
        
        Note that $\tilde{r}(0) = r(0)$, $\tilde{r}(t) = r(t) + h$. Next, we define $A$:
        \begin{equation*}
            A(h) = u(t, r(t) + h) - \int_0^t L(r + \alpha h, \dot{r} + \dot{\alpha} h) + F(r + \alpha h, \gamma_T^{\tau}[X]) \, d\tau - G(0, \gamma_T^0[X])
        \end{equation*}
        
        See that $A(h) \leq 0 \;\, \forall h$. Particularly, $A(\mathbf{0}) = 0 \implies h = \mathbf{0}$ is a maximizer of $A$. Therefore, $\nabla A(\mathbf{0}) = 0$, which implies
        \begin{equation*}
            0 = \nabla_q u(t, r(t)) - \int_0^t \langle \nabla_q L(r, \dot{r}), \alpha \rangle + \langle \nabla_v L(r, \dot{r}), \dot{\alpha} \rangle + \langle \nabla F(r, \gamma_T^{\tau}[X]), \alpha \rangle \, d\tau
        \end{equation*}
        
        Now note that via Integration by Parts,
        \begin{align*}
            \int_0^t \langle \nabla_v L(r, \dot{r}), \dot{\alpha} \rangle \, d\tau =& \biggl.\langle \nabla_v L(r,\dot{r}), \alpha \rangle \biggr]_0^t - \int_0^t \left\langle \frac{d}{d\tau} \nabla_v L(r,\dot{r}), \alpha \right\rangle \, d\tau \\
            &= \nabla_v L(r(t), \dot{r}(t)) + \int_0^t \left\langle -\frac{d}{d\tau} \nabla_v L(r, \dot{r}), \alpha \right\rangle \, d\tau
        \end{align*}
        
        Hence
        \begin{equation*}
            0 = \nabla_q u(t, r(t)) - \nabla_v L(r(t), \dot{r}(t)) + \int_0^t \left\langle \nabla_q L(r, \dot{r}) - \frac{d}{d\tau} \nabla_v L(r,\dot{r}) + \nabla F(r, \gamma_T^{\tau}[X]), \alpha \right\rangle \, d\tau
        \end{equation*}
        
        By (\ref{euler_lagrange_eq_u}) (as $r(s) = \gamma_t^s[X]$),
        \begin{equation*}
            \int_0^t \left\langle \nabla_q L(r, \dot{r}) - \frac{d}{d\tau} \nabla_v L(r,\dot{r}) + \nabla F(r, \gamma_T^{\tau}[X]), \alpha \right\rangle \, d\tau = 0
        \end{equation*}
        
        The result follows.
    \end{proof}
    
    Thus, (as every minimizer of $u$ is also a minimizer of $\hat{U}$) by Lemma \ref{cont_of_seq_of_gamma} and Claim \ref{u_hje_claim_1} we have that as $r(t) = \gamma^t_t[X](\omega)$, if $t_n \to t$,
    \begin{align*}
        \lim_{n \to \infty} \nabla_q u(t_n, q) &= \lim_{n \to \infty} \nabla_v L(r(t_n), \dot{r}(t_n)) \\
        &= \nabla_v L(r(t), \dot{r}(t)) \\
        &= \nabla_q u(t,q)
    \end{align*}
    
    This implies that $\nabla_q u$ is also continuous in the time input, which concludes the proof.
\end{proof}

We use the existence and continuity of $\nabla_q u$ to establish the existence and continuity of $\partial_t u$.

\begin{theorem}
    \label{first_derivative_t_u}
    Given the existence, continuity of $\nabla_q u(t,q)$, then $\partial_t u(t,q)$ exists and is continuous.
\end{theorem}
\begin{proof}
    First we note that, if $q = X(\omega)$, for $r(s) = \gamma_t^s[X](\omega)$, $r(t) = q$ and thus
    \begin{equation*}
        u(t, r(t)) - u(t - h, r(t - h)) = \int_{t - h}^t L(r, \dot{r}) + F(r, \gamma_T^\tau[X]) \, d\tau
    \end{equation*}
    
    Our goal is to show that the following limit exists:
    \begin{equation*}
        \lim_{h \to 0} \frac{u(t, r(t)) - u(t - h, r(t))}{-h}
    \end{equation*}
    
    We note that by adding and subtracting $u(t - h, r(t))$
    \begin{equation*}
        \left[u(t, r(t)) - u(t - h, r(t)\right] + \left[u(t - h, r(t)) - u(t - h, r(t - h))\right] = \int_{t - h}^t L(r, \dot{r}) + F(r, \gamma_T^\tau[X]) \, d\tau
    \end{equation*}
    
    We divide by $-h$ on both sides, rearrange and obtain:
    \begin{equation}
        \label{u_difference_quotient}
        \frac{u(t, r(t)) - u(t - h, r(t)}{-h} = \frac{u(t - h, r(t)) - u(t - h, r(t - h))}{h} - \frac{1}{h} \int_{t - h}^t L(r, \dot{r}) + F(r, \gamma_T^\tau[X]) \, d\tau
    \end{equation}
    
    Using Taylor's Theorem,
    \begin{multline*}
        \frac{u(t - h, r(t)) - u(t - h, r(t - h))}{h} = \int_0^1 \biggl\langle \nabla_q u(t - h, r(t - h) + s(r(t) - r(t - h))), \\
        \frac{r(t) - r(t - h)}{h} \biggr\rangle \, ds
    \end{multline*}
    
    In which, if $h \to 0$, by the continuity of $\nabla_q u$ we have that
    \begin{align*}
        \lim_{h \to 0} \frac{u(t - h, r(t)) - u(t - h, r(t - h))}{h} &= \int_0^1 \langle \nabla_q u(t, r(t)), \dot{r}(t) \rangle \, ds \\
        &= \langle \nabla_q u(t, r(t)), \dot{r}(t) \rangle
    \end{align*}
    
    And also note
    \begin{equation*}
        \lim_{h \to 0} -\frac{1}{h} \int_{t - h}^t L(r, \dot{r}) + F(r, \gamma_T^\tau[X]) \, d\tau = -L(r(t), \dot{r}(t)) - F(r(t), \gamma_T^t[X])
    \end{equation*}
    
    Hence, letting $h \to 0$ in (\ref{u_difference_quotient}) gives us
    \begin{equation*}
        \lim_{h \to 0} \frac{u(t, r(t)) - u(t - h, r(t))}{-h} = \langle \nabla_q u(t, r(t)), \dot{r}(t) \rangle - L(r(t), \dot{r}(t)) - F(r(t), \gamma_T^t[X])
    \end{equation*}
    
    Hence we conclude the limit exists. By definition, this limit simply differentiates with respect to time, and thus we have that $\partial_t u$ exists. Furthermore, it is defined as a composition of continuous functions, therefore it is continuous, which concludes the proof.
\end{proof}

\subsection{The Hamilton--Jacobi Equation for \texorpdfstring{$u$}{u}}

We now proceed to demonstrate that $u$ satisfies the Hamilton--Jacobi Equation. 

\begin{proposition}
    \label{u_hje}
    Given the existence of $\partial_t u(t,q), \nabla_q u(t,q) \in C$, we have
    \begin{equation*}
        \partial_t u(t,q) + H(q, \nabla_q u(t,q)) - \int_{\Omega} \Phi(q - \gamma_T^t[X](\Tilde{\omega})) \, d\Tilde{\omega} = 0
    \end{equation*}
    
    With the boundary condition
    \begin{equation*}
        u(0,q) = \int_{\Omega} \Psi(q - \gamma_T^0[X](\tilde{\omega})) \, d\tilde{\omega}
    \end{equation*}
\end{proposition}

\begin{proof}
    From Proposition \ref{euler_lagrange_thm_u}, we may consider the Euler--Lagrange Equation of minimizer $s \mapsto \gamma_t^s[X]$ achieving minimum $u(t,q)$:
    \begin{equation}
        \label{euler_lagrange_eq_u}
        \frac{d}{ds} \nabla_v L(\gamma_t^s[X](\omega), \dot{\gamma}_t^s[X](\omega)) = \nabla_q L(\gamma_t^s[X](\omega), \dot{\gamma}_t^s[X](\omega)) + \int_{\Omega} \nabla \Phi(\gamma_t^s[X](\omega) - \gamma_T^s[X](\Tilde{\omega})) \, d\Tilde{\omega}
    \end{equation}

    Again, we note that $\nabla_v L(q, \cdot)$ and $\nabla_p H(q, \cdot)$ are inverses of each other. Set 
    \begin{equation}
        \label{eqn_z_hsys_u}
        Z(s, \omega) = \nabla_v L(\gamma_t^s[X](\omega), \dot{\gamma}_t^s[X](\omega))
    \end{equation}
    
    Giving us
    \begin{equation}
        \label{eqn_gammadot_hsys_u}
        \dot{\gamma}_t^s[X](\omega) = \nabla_p H(\gamma_t^s[X](\omega), Z(s,\omega))
    \end{equation}
    
    Recall Lemma \ref{uhat_hje_lemma_2}, and notice (\ref{eqn_z_hsys_u}) implies
    \begin{equation*}
        L(\gamma_t^s[X](\omega), \dot{\gamma}_t^s[X](\omega)) + H(\gamma_t^s[X](\omega), Z(s,\omega)) = \langle \dot{\gamma}_t^s[X](\omega), Z(s, \omega) \rangle
    \end{equation*}
    
    Which, when combined with Lemma \ref{uhat_hje_lemma_1}, yields
    \begin{equation}
        \label{lagrange_to_hamiltonian_u}
        \nabla_q L(\gamma_t^s[X](\omega), \dot{\gamma}_t^s[X](\omega)) = -\nabla_q H(\gamma_t^s[X](\omega), Z(s,\omega))
    \end{equation}
    
    Combining (\ref{euler_lagrange_eq_u}), (\ref{eqn_z_hsys_u}), (\ref{eqn_gammadot_hsys_u}), and (\ref{lagrange_to_hamiltonian_u}), we have the Hamiltonian system
    \begin{equation*}
        \begin{dcases}
            \dot{Z}(s, \omega) = -\nabla_q H(\gamma_t^s[X](\omega), Z(s, \omega)) + \int_{\Omega} \nabla \Phi(\gamma_t^s[X](\omega) - \gamma_T^s[X](\tilde{\omega})) \, d\tilde{\omega} \\
            \dot{\gamma}_t^s[X](\omega) = \nabla_p H(\gamma_t^s[X](\omega), Z(s, \omega))
        \end{dcases} 
    \end{equation*}
    
    Now note that (for $\bar{r}(t) = q$)
    \begin{equation*}
        u(t, \bar{r}(t)) \leq \int_0^t L(\bar{r}(s),  \dot{\bar{r}}(s)) + F(\bar{r}(s), \gamma_T^s[X]) \, ds + G(0, \gamma_T^0[X]) \quad \forall \bar{r}
    \end{equation*}
    
    Rearranging,
    \begin{equation*}
        u(t, \bar{r}(t)) - \int_0^t L(\bar{r}(s),  \dot{\bar{r}}(s)) + F(\bar{r}(s), \gamma_T^s[X]) \, ds - G(0, \gamma_T^0[X]) \leq 0 \quad \forall \bar{r}
    \end{equation*}
    
    Let $\bar{r} = r$ be optimal in $u$ ($r(s) = \gamma_t^s[X](\omega)$). Thus,
    \begin{equation*}
        a(\tau) \overset{\text{set}}{=} u(\tau, r(\tau)) - \int_0^{\tau} L(r(s), \dot{r}(s)) + F(r(s), \gamma_T^s[X]) \, ds - G(0, \gamma_T^0[X]) = 0
    \end{equation*}
    
    Therefore $a'(\tau) = 0$ and
    \begin{equation}
        \label{derivative_a}
        \partial_\tau u(\tau, r(\tau)) + \langle \nabla_q u(\tau, r(\tau)), \dot{r}(\tau) \rangle - L(r(\tau), \dot{r}(\tau)) - F(r(\tau), \gamma_T^{\tau}[X]) = 0
    \end{equation}
    
    Combining (\ref{derivative_a}), Lemma \ref{uhat_hje_lemma_2}, and Claim \ref{u_hje_claim_1} we conclude
    \begin{equation*}
        \partial_{\tau} u(\tau, r(\tau)) + H(r(\tau), \nabla_q u(\tau, r(\tau)) - F(r(\tau), \gamma_T^{\tau}[X]) = 0 \quad \forall \tau
    \end{equation*}
    
    Set $\tau = t$. Then $r(t) = q$ and we arrive at the Hamilton--Jacobi Equation for $u(t,q)$:
    \begin{equation*}
        \partial_t u(t,q) + H(q, \nabla_q u(t,q)) - \int_{\Omega} \Phi(q - \gamma_T^t[X](\tilde{\omega})) \, d\tilde{\omega} = 0
    \end{equation*}
    
    We simply plug in $t = 0$ into $u$ to obtain the boundary condition
    \begin{equation*}
        u(0,q) = G(r(0), \gamma_T^0[X]) = \int_{\Omega} \Psi(q - \gamma_T^0[X](\tilde{\omega})) \, d\tilde{\omega} \qedhere
    \end{equation*}
\end{proof}
\section{Conclusion}

Thus, as we demonstrated the existence of $u: [0,T] \times \mathbb{R}^d \to \mathbb{R}$, $\gamma: [0,T] \times \Omega \to \mathbb{R}^d$ of class $C^1$ with $\nabla_{qq} u$ bounded such that
\begin{equation*}
    \begin{dcases}
        \partial_t u(t,q) + H(q, \nabla_q u(t,q)) - \int_{\Omega} \Phi(q - \gamma_T^t[X](\tilde{\omega})) \, d\tilde{\omega} = 0 \\
        \dot{\gamma}_T^t[X](\omega) = \nabla_p H(\gamma_T^t[X](\omega), \nabla_q u(t, \gamma_T^t[X](\omega)) \\
        u(0,q) = \int_{\Omega} \Psi(q - \gamma_T^0[X](\tilde{\omega})) \, d\tilde{\omega} \\
        \gamma_T^T[X](\omega) = X(\omega)
    \end{dcases}
\end{equation*}

As we assumed $X$ is non-atomic, by Theorem \ref{primary_thm} we conclude that there exists a unique Nash Equilibrium.
\printbibliography

\end{document}